\documentclass[11pt, twoside, a4paper]{amsart}

\usepackage{amsfonts, amsmath, amssymb, amsthm, mathtools,hyperref}
\usepackage{bm}
\newtheorem{theorem}{Theorem}

\newtheorem{remark}{Remark}
\newtheorem{lemma}{Lemma}

\newcommand{\R}{\mathbb{R}}

\numberwithin{equation}{section}
\newcommand{\fin}{\hfill\rule{2mm}{2mm}\medskip }


\title[Controllability of coupled parabolic systems]{Internal controllability of  parabolic systems with star and tree like couplings}
\author[C.G. Lefter]{C\u at\u alin-George Lefter}
\author[E.A. Melnig]{Elena-Alexandra Melnig}
\address[C.G. Lefter, E.A. Melnig]{ Octav Mayer Institute of Mathematics, Romanian Academy, Ia\c si Branch and Faculty of Mathematics, University "Al. I. Cuza" Ia\c si, Romania}
\email{catalin.lefter@uaic.ro, alex.melnig@yahoo.com}
\date{}

\begin{document}

\begin{abstract}{We consider systems of  parabolic equations coupled in zero order terms in a star-like or a  tree-like shape,  with an internal control acting in only one of the equations. We obtain local exact  controllability  to the stationary solutions of the  system,  under hypotheses concerning the supports of the coupling functions. The key point is establishing  Carleman estimates with appropriate observation operators for the adjoint to the linearized system, which allows the study of the controllability problem, in either linear or nonlinear cases, in an $L^\infty$ framework.}
\end{abstract}
\maketitle

\section{Introduction}
In this paper we consider some classes of semilinear systems of parabolic equations coupled  in zero order terms. 
We are interested in controllability of such systems to  stationary solutions by only one scalar control distributed in a subdomain and acting in only one of the equations. 

The study of  controlled systems of parabolic equations needs appropriate observability estimates for the adjoint system. These observability estimates are usually derived from global Carleman estimates.
Global Carleman estimates are by now a classical tool in proving observability inequalities, and they were established in the context of controllability  for parabolic equations by  O.Yu.Imanuvilov (see O.Yu.Imanuvilov and A.Fursikov  \cite{furima1996}). Since then this type of estimates was  extensively developed, refined and used in other contexts, like control problems with small number of controls, stabilization or inverse problems.

Controllability for parabolic systems with a reduced number of controls needs observability estimates of Carleman type  with  partial observations. There is an extensive literature concerning such problems; for a selection of titles we refer for example to \cite{lissy_zuazua} and the references therein.

In the case of zero order couplings with constant or time dependent coupling coefficients there exists a particular interest in obtaining algebraic conditions of Kalman type for controllability; in this direction we cite the papers of  F.Ammar-Khodja, A.Benabdallah, C.Dupaix and M.Gonz\'ales-Burgos \cite{khodja_burgos1, khodja_burgos2} or the work of F.Ammar-Khodja, F.Chouly and M. Duprez \cite{ammar_chouly}.  

Observability estimates for linear systems (not only parabolic) coupled with constant coupling coefficients in the dominant part and/or in the zero order terms were established by E.Zuazua and P.Lissy \cite{lissy_zuazua}; such estimates are obtained under Kalman rank conditions satisfied by the pair of the coupling and  control matrices.  

The results we present in our paper extend the results in \cite{teresa_burgos} where systems of parabolic equations with cascade type couplings in zero order terms are considered. The extension we propose works under hypotheses addressing two aspects of the systems under consideration: one is  the structure of the 
couplings, which describes in our case  either a star or a tree type graph; the second aspect refers to  the support of the coupling functions or, in the linear case, to the support of the coupling coefficients.

The strategy for proving  the controllability result relies on the linearization of the nonlinear system around a stationary state. 
The key step is 
 obtaining the null controllability for this linear system by using an observability inequality for the adjoint system. This observability inequality is consequence of an appropriate global Carleman estimate. This in turn is obtained by combining Carleman estimates for each of the equations, but relying on different auxiliary functions, which are in a particular order relation, made possible  by the special structure of the system. The idea of using different auxiliary functions in Carleman estimates is inspired by the work of G.Olive \cite{olive1} concerning controllability of parabolic systems with controls acting in different subdomains.

 The Carleman observability estimates we establish are more elaborated and are not direct consequences of the classical Carleman estimates. One reason for developing these Carleman estimates is the fact that trying to use the estimates from the paper of Luz de Teresa and M.Gonz\'ales-Burgos \cite{teresa_burgos} for cascade systems we realized that, written for the branches of the tree, they do not fit well together. Even when passing from the study of star type couplings to general tree type couplings one needs to use two Carleman estimates for each equation in an interior node of the graph and this is another quite technical point needed in our approach. The hypotheses concerning the supports of the coupling coefficients allow to construct appropriate auxiliary functions and  weights in the corresponding Carleman estimates which will finally fit well in order to give the desired global observability inequality.

Passing from the linearized system to the nonlinear system, one  needs an $L^\infty$ framework for controlability. The main reason is that the Carleman estimates we obtain are sensitive to zero order perturbations of the system. 
More regularity of the controls in the linearized problem is obtained as in the work of V.Barbu \cite{bar12002} (see also J.-M. Coron, S.Guerrero and L.Rosier \cite{corgueros2010}) by using regularizing properties of the parabolic flow in a bootstrap argument.  This step encounters supplementary technical challenges as it needs $L^\infty-L^2$ Carleman estimates with different weights in corresponding estimates for different equations. 
The $L^\infty$ controllability for the linearized system allows an approach to the controllability of the  nonlinear system by a fixed point argument, based on Kakutani theorem ( see also \cite{corgueros2010} or \cite{bar2000}).

%The coupling is done in such a way that the function evolving under the constrain of the control acts in every other equation in a %subdomain, such that the control acts in an indirect manner in all equations in such compatibility subdomains.
%For example, the choice of the weight functions in the paper of E. Zuazua, E.Fernandez-Cara, \cite{ferzua2000} is important in %obtaining optimal bounding constants in estimating the cost of controllability.
%%%%%%%%%%%%%%%%%%%%%%%%%%%%%%%%%%%%%%%%%%%%%%%%%%%%%%%%%%%%%%%%%%%%%%%%%

\section{Preliminaries and statement of the problem}
Let $\Omega\subset\R^N$ be a bounded connected domain with a $C^2$
boundary $\partial\Omega$ and let 
$\omega_0\subset\subset\Omega.$  
Let $T>0$ and denote by $Q=(0,T)\times\Omega$ and for $\omega\subset\Omega$ write $Q_\omega=(0,T)\times\omega$.
\medskip

We consider   systems of  $(n+1)$ parabolic equations coupled in zero order terms through nonlinear functions,  with one internally distributed control, acting in $\omega_0$ and entering only the first equation. The main goal  is obtaining   local exact controllability to some stationary solution for the nonlinear system.
\medskip

In the first part of the paper we study systems of parabolic equations with star-like couplings which refer to the sistuation where $y_k$ is actuated in the corresponding parabolic equation  through a nonliniarity depending  only  on $y^0,y^k$. Such a star-like coupled system has the form:
\begin{equation}\label{nonlinsystem}
\left\lbrace
\begin{array}{ll}
D_ty_{0}-\Delta y_0=\overline g_0(x)+f_0(x,y_0)+\chi_{\omega_0}u, & \text{ in }(0,T)\times\Omega,\\
D_ty_{i}-\Delta y_i=\overline g_i(x)+f_i(x,y_0,y_i),\, i\in\overline{1,n},\, &\text{ in }(0,T)\times\Omega,\\
y_0=...=y_{n}=0, &\text{ on }(0,T)\times\partial\Omega,\\
y(0,\cdot)=y^0,&
\end{array}
\right.
\end{equation} 
where $\overline g_j\in L^\infty(\Omega),\,  j\in\overline{0,n}$. 
We denote by  $\chi_{\omega_0}v$  the extension of $v:\omega_0\rightarrow\R$ with $0$ to the whole domain $\Omega$.
The control function is $u:[0,T]\times\omega_0\longrightarrow\R$, acting directly in the equation of $y_0$ while the other components of the solution, $y_1,...,y_n$, are indirectly actuated through the corresponding coupling terms containing  $y_0$.
\medskip

Consider a  stationary state  $\overline y=(\overline{y}_0,...,\overline{y}_n), \overline{y}_j\in L^\infty(\Omega), j\in\overline{0,n}$, solution to the elliptic system: 
\begin{equation}\label{nonlinelliptic}\left\lbrace
\begin{array}{ll}
-\Delta \overline y_0=\overline g_0(x)+f_0(x,\overline y_0), & x\in\Omega,\\
-\Delta \overline y_i=\overline g_i(x)+f_i(x,\overline y_0,\overline y_i),\, i\in\overline{1,n},\, &x\in\Omega,\\
\overline y_0=...=\overline y_{n}=0, &x\in\partial\Omega.\\
\end{array}
\right.
\end{equation}
Observe in fact that, by elliptic regularity, an $L^\infty$ stationary solution is a smooth solution. 
\medskip

Concerning the coupling terms we assume the following hypotheses:\medskip

\begin{enumerate}
	\item[\textit{(H1)}] $f_0:\R^N\times\R\longrightarrow\R$, $ f_i:\R^N\times\R\times\R\longrightarrow\R, i\in\overline{1,n} $ are $C^1$ functions and there exist $\omega_1,...\omega_{n}\subset \Omega$,  open nonempty subsets of
	$\Omega$ such that
	\begin{equation}
	(\omega_i\cap\omega_0)\setminus\bigcup_{j\ne0,i}\omega_j\neq\emptyset,\,\forall i\in\overline{1,n},
	\end{equation}
and for all $i\in\overline{1,n}$ we have
	\begin{equation}\label{fsuport}
	f_i(x,y_0,y_i)=0 \,\forall x\in\Omega\setminus\omega_i,\, y_0,y_i\in\R;
	\end{equation}
\item[\textit{(H2)}]The  following coupling condition holds:
\begin{equation}\label{suportderiv}
\text{supp }\frac{\partial f_i}{\partial y_0}(x,\overline{y}_0(x),\overline{y}_i(x))\cap\biggl\{(\omega_i\cap\omega_0)\setminus\overline{\bigcup_{j\ne0,i}\omega_j}\biggr\}\neq\emptyset.
\end{equation}
\end{enumerate}

\begin{remark}
Concerning the above technical hypotheses \eqref{fsuport}, \eqref{suportderiv}-(H2), observe that they are, for example, satisfied for all sources $\overline g_i$ and corresponding stationary solutions $\overline y$ if  the nonlinearities $f_i$ are of the form
$$
f_i(x,y_0,y_i)=\zeta_i(x)\xi_i(y_0,y_i),\,x\in\Omega,y_0,y_i\in\R,
$$
with $\emptyset\not=\mbox{supp }\zeta_i(x)\subset \subset(\omega_i\cap\omega_0)\setminus\overline{\bigcup_{j\ne0,i}\omega_j}$ and $\xi_i=\xi_i(y_0,y_i):\mathbb{R}\times\mathbb{R}\rightarrow\mathbb{R}$ are smooth with $\displaystyle\frac{\partial \xi_i}{\partial y_0}\ne0,\,\forall y_0,y_i\in\R$.

\noindent If $\overline y$ is a constant solution, as the problem has homogeneous boundary conditions, necessarily $\overline y\equiv 0$; a stationary solution in this case exists if and only if $\overline g_0(x)=-f_0(x,0)$ and $\overline g_i(x)=-f_i(x,0,0),\forall x\in\Omega$. Condition \eqref{suportderiv} is satisfied if, for example, $\displaystyle\emptyset\ne\text{supp }\frac{\partial f_i}{\partial y_0}(x,0,0)\subset\subset (\omega_i\cap\omega_0)\setminus\overline{\bigcup_{j\ne0,i}\omega_j}.$

\noindent 

In concrete situations, when the stationary solution is known, the hypotheses we imposed on the supports of the coupling functions   are easy to verify. 
\end{remark}

Our study concerns the controllability to the stationary state $\overline y$  of system \eqref{nonlinsystem}  in a given time interval $T$.  We are thus led to the study of a class of linear  controlled systems and corresponding controllability properties, systems which arise by a linearization procedure around the stationary state:
\begin{equation}\label{linsystem}
\left\lbrace
\begin{array}{ll}
D_tz_{0}-\Delta z_0=c_0(t,x)z_0+\chi_{\omega_0}u, & (0,T)\times\Omega,\\
D_tz_{i}-\Delta z_i=a_{i0}(t,x)z_0+c_i(t,x)z_i,\, i\in\overline{1,n},\, &(0,T)\times\Omega,\\
z_0=...=z_{n}=0, &(0,T)\times\partial\Omega,\\
\end{array}
\right.
\end{equation} 
For $M,\delta>0,$ and open subsets $\underline{\omega_i}\subset\subset(\omega_i\cap\omega_0)\setminus\bigcup_{j\ne0,i}\omega_j $ we introduce the following classes  of coefficients sets:
\begin{equation}\label{hyp}
\begin{aligned}
&\mathcal{E}_{M,\delta,\{\underline\omega_i\}_i}=\biggl\{ E=\{a_{i0},c_j\}_{i\in\overline{1,n},j\in\overline{0,n}}:a_{i0},c_j\in L^\infty(Q),  \\
&\|a_{i0}\|_{L^\infty},\|c_j\|_{L^\infty}\leq M, a_{i0}=0 \text{ in } Q\setminus Q_{\omega_i},\text{ and } |a_{i0}|\ge\delta \text{ on }Q_{\underline\omega_i}\biggr\}.
\end{aligned}
\end{equation}
We prove first that such linear systems with coefficients in $\mathcal{E}_{M,\delta,\{\underline\omega_i\}_i}$ are null controllable with norm $L^2$ and $L^\infty$ of the control uniformly bounded by a constant $C=C(M,\delta,\{\underline\omega_i\}_i)$.

In order to achieve this goal we consider the adjoint system:
\begin{equation}\label{adjlinsystem}
\left\lbrace
\begin{array}{ll}
-D_tp_{0}-\Delta p_0=c_0(t,x)p_0+\sum_{i=1}^{n}a_{i0}(t,x)p_i, & (0,T)\times\Omega,\\
-D_tp_{i}-\Delta p_i=c_i(t,x)p_i,\, i\in\overline{1,n},\, &(0,T)\times\Omega,\\
p_0=...=p_{n}=0, &(0,T)\times\partial\Omega.\\
\end{array}
\right.
\end{equation}
We prove an observability inequality as consequence of an  appropriate Carleman estimate.
The Carleman estimate we establish in the next section gives us more than just observability, it helps obtaining a priori estimates for the control driving the solution of the linear system to zero and, as the constants appearing in the Caleman estimates are depending only on $M,\delta,\{\underline\omega_i\}_i$, the estimates on the control will result uniform. 
This fact is essential in the fixed point argument when dealing with the nonlinear system.

In order to reformulate the problem in an abstract functional framework let the state space be  the Hilbert space $H=[L^2(\Omega)]^{n+1}$ and the control space $U=L^2( \omega_0)$.
Consider the operator 
$$\textbf{A}:D(\mathbf{A})\subset H\longrightarrow H, D(\textbf{A})=(H_0^1(\Omega)\cap
H^2(\Omega))^{n+1}, \textbf{A}z=\Delta z,$$
and the control operator  $$\textbf{B}:U\rightarrow H,\, \textbf{B}u=\chi_{\omega_0}Bu,\, B=(1,0,\ldots,0)^\top.$$
Then, problem \eqref{nonlinsystem} may be written in abstract form:
\begin{equation}\label{nonlinevpb}
\left\lbrace
\begin{array}{ll}
D_ty=\textbf{A}y+\textbf{f}(y) +\textbf{B}u, & t>0,\\
y(0)=y^0.   & \,\\
\end{array}
\right.
\end{equation} where $\textbf{f}(y)=f(\cdot,y(\cdot))$.
The linear   problem \eqref{linsystem} may be reformulated as: 
\begin{equation}\label{linevpb}
\left\lbrace
\begin{array}{ll}
D_tz=\textbf{A}z+\mathbf{A_0}(t)z+\mathbf{C}(t)z+\textbf{B}u, & t>0,\\
z(0)=z^0,   &\\
\end{array}
\right.
\end{equation} 
where $\textbf{C}(t)z=C_0(t,\cdot)z(\cdot)$ and $\mathbf{A_0}(t)z=A_0(t,\cdot)z(\cdot)$,  $ C_0(t,x)$ is the diagonal matrix $C_0(t,x)=diag(c_i(t,x))_{i=\overline{0,n}}$ and the coupling matrix  $A_0(t,x)$ has  only one  nonzero column,  the first one, and is given by
$$A_0(t,x)=(0,a_{10},\ldots,a_{n0})^\top\cdot (1,0,\ldots,0).$$
 
 For simplicity, when there is no confusion, we denote the norms of functions   $z\in[L^2(\Omega)]^{n+1}, $ $z\in [H^1(\Omega)]^{n+1}$  \textit{etc.} as $\|z\|_{L^2(\Omega)}$,  respectively $\|z\|_{H^1(\Omega)}$  \textit{etc.}. 
 
Null controllabity for the linear system  \eqref{linevpb} above is equivalent to an observability inequality 
\begin{equation}\label{obsineq0}
\|p(0)\|^2_{L^2(\Omega)}\leq C(M,\delta)\int_0^T
\|\textbf{B}^* p\|_{L^2(\omega_0)}^2dt,  \qquad \text{ for some } C(M,\delta)>0,
\end{equation}
for all solutions $p$ to the adjoint equation
\begin{equation}\label{adjlinevpb}
-p'=\textbf{A}p+\mathbf{A_0^*} p+\textbf{C}p
\end{equation} 
where $\mathbf{A_0^*} p=A_0^\top p, \textbf{B}^*p=B^\top p|_{\omega_0}$.
\medskip

We extend our study to parabolic systems with tree-like couplings. In fact we will treat only linear equations with appropriate hypotheses for the coupling coefficients in a tree-like structure. Passing from linear results of controllability to local controllability for nonlinear systems may be obtaiend by exactly the same procedure as in the star-like case. An example of linear parabolic system with tree-like couplings is the following:
\begin{equation}\label{linsystem_tree}
\left\lbrace
\begin{array}{ll}
D_tz_{0}-\Delta z_0=c_0(t,x)z_0+\chi_{\omega_0}u, & \text{in }(0,T)\times\Omega,\\
D_tz_{1}-\Delta z_1=a_{10}(t,x)z_0+c_1(t,x)z_1,\,  &\text{in }(0,T)\times\Omega,\\
D_tz_{2}-\Delta z_2=a_{20}(t,x)z_0+c_2(t,x)z_2,\,  &\text{in }(0,T)\times\Omega,\\
D_tz_{3}-\Delta z_3=a_{31}(t,x)z_1+c_3(t,x)z_3,\,  &\text{in }(0,T)\times\Omega,\\
D_tz_{4}-\Delta z_4=a_{41}(t,x)z_1+c_2(t,x)z_4,\,  &\text{in }(0,T)\times\Omega,\\
z_0=...=z_{4}=0, &\text{on }(0,T)\times\partial\Omega,\\
\end{array}
\right.
\end{equation} 
and the general form of system with tree like couplings will be discussed in \S \ref{sec-tree}.\medskip 

The paper is organized as follows:

\begin{itemize}
	\item In \S \ref{secobservcarl} we prove appropriate Carleman estimates for adjoint system \eqref{adjlinsystem} in either $L^2-L^2$ or $L^\infty-L^2$ settings. This will be Theorem \ref{th1obs}
	\item In \S \ref{seccontrtreelin} we prove the null controllability of linear system \eqref{linsystem}. The approach uses  a family of optimal control problems with penalized final cost. One then obtains besides controllability an estimate for the control in both $L^2$ and $L^\infty$ norms by using the previous Carleman estimates. This is Theorem \ref{thcontrol}.
		\item \S \ref{seccontrtreenonlin} is devoted to the local controllability in $L^\infty$ of nonlinear system \eqref{nonlinsystem}. The fact that controllability has to be proved in $L^\infty$ is due to the high sensitivity of the Carleman estimates with respect to the coupling coefficients, which is not the case when controls act in each equation of the system.  The argument is similar to that used in \cite{corgueros2010}.
		\item In \S\ref{sec-tree} we extend results of controllability, with one distributed scalar control, for linear systems of parabolic equations, of the form \eqref{linsystem_tree}, with tree-like couplings. The key point here is obtaining appropriate Carleman estimates. Local controllability for nonlinear systems with tree-like couplings is also discussed.  
		
\end{itemize}

\section{ Carleman estimates and observability}\label{secobservcarl}
In this section we establish an $L^2$ Carleman estimate that will help proving an observability inequality fot the adjoint problem \eqref{adjlinsystem}. This $L^2$ Carleman inequality and parabolic regularity  are the starting point in obtaining an $L^\infty$ control through a bootstrap argument.

We recall the classical Carleman estimate for a generic nonhomogeneous parabolic problem, 
\begin{equation}\label{genericparab}
\left\lbrace
\begin{array}{ll}
D_tp+Lp=h, & \text{ in } (0,T)\times\Omega,\\
p=0,   &\text{ on } (0,T)\times\partial\Omega,\\
\end{array}
\right.
\end{equation} 
where $L$ is an uniformly elliptic operator of second order.  Denote by $Q:=(0,T)\times\Omega$ and, for  $\omega\subset\subset \Omega$, $ Q_{\omega}:=(0,T)\times\omega$.
 The solution is observed in $Q_\omega$ for  sources $ h\in L^2(Q)$.

 We introduce the function
$$\psi\in C^2(\overline{\Omega}),\, \psi|_{\partial\Omega}=k>0,\, k<\psi<\frac32k \text{ in }\Omega, \{x\in\overline\Omega: |\nabla\psi(x)|=0\}\subset\subset\omega, $$
and the weight functions
\begin{equation}\label{fialfagen}
\varphi(t,x):=\frac{e^{\lambda\psi(x)}}{t(T-t)},\quad
\alpha(t,x):=\frac{e^{\lambda\psi(x)}-e^{1.5\lambda\|\psi\|_{C(\overline{\Omega})}}}{t(T-t)}.
\end{equation}
Then, the classical global Carleman estimate (see \cite{furima1996}, \cite{ferzua2000}) is the following:

\begin{lemma}\label{lemaCarleman} There exist $\lambda_0,s_{0}$ and $C>0$ such that if $\lambda>\lambda_0, s\geq s_0$, the following inequality holds:
	\begin{equation}\label{estCarleman}\begin{aligned}
	&\int_Q\left[(s\varphi)^{-1} (|D_tp|^2+|D^2p|)+s\lambda^{2}\varphi|D p|^2+s^{3}\lambda^{4}\varphi^{3}|p|^2\right]e^{2s\alpha}dxdt\\
	&\leq C\int_{Q_\omega}s^{3}\lambda^{4}\varphi^{3}|p|^2e^{2s\alpha}dxdt+\int_{Q}|h|^2 e^{2s\alpha}dxdt\\
	\end{aligned}
	\end{equation}
	for all $p\in H^1(0,T;L^2(\Omega))\cap L^2(0,T; H^2(\Omega))$ solution of \eqref{genericparab}.
\end{lemma}

We establish a  Carleman estimate for the following nonhomogeneous version to the adjoint problem to  \eqref{linsystem}, with  source term  $g\in [L^2(Q)]^{n+1}$ and observation operator $\textbf{B}^*p=B^\top p|_{\omega_0}=p_0|_{\omega_0}$, operator which "sees" only $p_0$ in  the subdomain $ \omega_0$:
\begin{equation}\label{withsource}
\left\lbrace
\begin{array}{ll}
-D_tp_{0}-\Delta p_0=c_0(t,x)p_0+\sum_{i=1}^{n}a_{i0}(t,x)p_i+g_0, & (0,T)\times\Omega,\\
-D_tp_{i}-\Delta p_i=c_i(t,x)p_i+g_i,\, i\in\overline{1,n},\, &(0,T)\times\Omega,\\
p_0=...=p_{n}=0, &(0,T)\times\partial\Omega.\\
\end{array}
\right.
\end{equation}
In the following we are going to establish Carleman estimates for each equation in \eqref{withsource} by using in each case corresponding subdomains of observation and appropriately chosen weight functions. 

Some technical preliminaries are needed and we proceed as follows:

Consider open subsets 

$$\tilde\omega_j\subset\subset\underline\omega_j$$ 
and denote as above  by $Q_{\tilde\omega_j}=(0,T)\times\tilde\omega_j$; take the   auxiliary functions $\psi_j, j=\overline{0,n}$,  with the following properties (where we have denoted by $\tilde\omega_0:=\omega_0$):
\begin{equation}\label{psi}
\psi_j:=\eta_j+K_j, j\in \overline{0,n},
\end{equation}
$$\eta_j\in C^2(\overline{\Omega}),\, 0<\eta_j \text{ in }\Omega,\quad\eta_j|_{\partial\Omega}=0,\quad \{x\in\overline\Omega: |\nabla\eta_j(x)|=0\}\subset\subset\tilde\omega_j,$$
for some fixed positive constants $K_j>0$  such that
\begin{equation}\label{psiipsi0}
	\psi_i>\psi_0 \text{ in } \Omega, \forall i\in\overline{1,n},
\end{equation}

and
\begin{equation}\label{psipsi}
\frac{\sup \psi_j}{\inf\psi_j}<\frac87, \forall j\in\overline{0,n}.
\end{equation}
  Let $0<\epsilon<\inf\psi_i,  i\in\overline{0,n}$ a small positive number and denote by
 	\begin{equation}
 	\overline{\psi}=\sup_{x\in\Omega}\sup_{j\in \overline{0,n}}\psi_j(x)+\epsilon,
 	\qquad
 	\underline{\psi}=\inf_{x\in\Omega}\inf_{j\in \overline{0,n}}\psi_j(x)-\epsilon.
 	\end{equation}
 	Introduce also, for parameters $s,\lambda>0$ the auxiliary functions:
 	\begin{equation}\label{fialfa}
 	\varphi_j(t,x):=\frac{e^{\lambda\psi_j(x)}}{t(T-t)},\quad
 	\alpha_j(t,x):=\frac{e^{\lambda\psi_j(x)}-e^{1.5\lambda\overline{\psi}}}{t(T-t)}, \forall j\in\overline{0,n}
 	\end{equation}
and
 	\begin{equation}\label{fialfabar}
 	\overline\varphi(t)=\overline\varphi^\lambda(t):=\frac{e^{\lambda\overline\psi}}{t(T-t)},\quad
 	\overline\alpha(t)=\overline\alpha^\lambda(t):=\frac{e^{\lambda\overline\psi}-e^{1.5\lambda\overline\psi}}{t(T-t)},
 	\end{equation}
 	\begin{equation}\label{fialfabar2}
 	\underline\varphi(t)=\underline\varphi^\lambda(t):=\frac{e^{\lambda\underline\psi}}{t(T-t)},\quad
 	\underline\alpha(t)=\underline\alpha^\lambda(t):=\frac{e^{\lambda\underline\psi}-e^{1.5\lambda\overline\psi}}{t(T-t)}.
 	\end{equation}
 
 \begin{remark}\label{ordineaponderilor}
 	
 	\begin{enumerate}
 		\item[(i)] As we are going to compare the various Carleman estimates  stated for each equation of the linear adjoint system, we will need  to compare the weights which are involved in thgose inequalities. For this purpose let us observe  that given  $m_0>0$ there exist $s_0=s_0(m_0),\lambda_0=\lambda_0(m_0)>0$ such that for all $s>s_0, \lambda>\lambda_0$, $|m|\le m_0$ and $  t\in(0,T)$, the following inequality holds:
 		
 		\begin{equation}\label{weightsorder}
 		e^{s\underline\alpha}\leq s^m\varphi_i^me^{s\alpha_i}\leq e^{s\overline\alpha},
 		\end{equation}
 		
 		\begin{equation}\label{weightsorder1}
 		e^{s\alpha_0}\leq s^m\varphi_i^me^{s\alpha_i}.
 		\end{equation}
 		\item[(ii)] Observe that if in \eqref{psi} we replace $K_i$ with $K_i+M$ with  the constant $M>0$ big enough, the above properties of the auxiliary functions remain valid and, moreover, we may assume that 
 		\begin{equation}
 		\frac{\overline\psi}{\underline\psi}\leq\frac{3}{2}.
 		\end{equation} 
 		This extra assumption implies that there exist $\bar s_0>0,\bar\lambda_0>0$ such that if $s>\bar s_0$ $\lambda>\bar\lambda_0$,
 		\begin{equation}
 		|D_t\varphi_i|\leq C\varphi_i^2,\quad |D_t\alpha_i| \leq C\varphi_i^2,\quad |D^2_t\alpha_i| \leq C\varphi_i^3.
 		\end{equation}
 		\item[(iii)]Observe that for $\lambda$ big enough, say $\lambda>\overline\lambda$, we have 
 		\begin{equation}\label{ordineaponderilor2}
 		\frac{\;\underline{\alpha^\lambda}\;}{\overline{\alpha^\lambda}}<2.
 		\end{equation}
 		Indeed, this is a consequence to the fact that  $\displaystyle\lim_{\lambda\rightarrow+\infty}\frac{\;\underline{\alpha^\lambda}\;}{\overline{\alpha^\lambda}}=1$, uniformly with respect to $(t,x)\in Q$.
 \end{enumerate}\end{remark}\medskip

 	In this section we prove the following Carleman estimate which has as consequence the appropriate observability inequality for the adjoint system \eqref{genericparab}.
 
 \begin{theorem}\label{th1obs} There exist constants $\lambda_0,s_{0}$ such that for $\lambda>\lambda_0$ there exists a constant $C>0$ depending on $(M,\delta,\{\underline\omega_i\}_i,\lambda)$, such that, for any $ s\geq s_0$, the following inequality holds:
 	\begin{equation}\label{estCarlemanstar}\begin{aligned}
 	&\int_Q(|D_tp|^2+|D^2p|^2+|D p|^2+|p|^2)e^{2s\underline\alpha}dxdt\\
 	&\leq C\int_{Q_{\omega_0}}|p_0|^2e^{2s\overline\alpha}dxdt+C\int_{Q}|g|^2e^{2s\overline\alpha}dxdt\\
 	\end{aligned}
 	\end{equation}
 	for all $p\in H^1(0,T;L^2(\Omega))\cap L^2(0,T; H^2(\Omega))$ solution of \eqref{withsource}.

 	Moreover, there exist $m_0\in\mathbb{N}$ and $\delta_1>0$  such that for the homogeneous adjoint system (\textit{i.e.} taking $g\equiv 0$), we have the following $L^\infty-L^2$ Carleman estimate
 	\begin{equation}\label{LinftyL2Carleman}
 	\|p e^{(s+m_0\delta_1)\underline{\alpha}}\|_{L^{\infty}(Q)}\leq C \|p_0 e^{s\overline\alpha}\|_{L^2(Q_{\omega_0})}.
 	\end{equation}
 	\end{theorem}

\begin{proof}

The second remark above is useful when obtaining Carleman estimates, since the weights here are slightly different with respect to those  used in    \cite{furima1996} or \cite{fercarzua2000}. However, this remark allows following the same lines of proof and  we may  write Carleman estimate \eqref{estCarleman} 
for  each equation $j\in\overline{0,n}$ with observation domain $\tilde\omega_j$ and auxiliary functions and weight functions $\psi_j,\varphi_j,\alpha_j$. Thus, there exist $s_{0}>0, C>0$ such that for any $ s\geq s_0$, the following inequalities hold:
\begin{enumerate}
	\item For $p_0$ we have

	\begin{equation}\label{estCarleman0}\begin{aligned}
	&\int_{Q}\left[(s\varphi_0)^{-1}(|D_tp_0|^2+|D^2p_0|^2)+s\varphi_0|Dp_0|^2+s^{3}\varphi_0^{3}|p_0|^2\right]e^{2s\alpha_0}dxdt\\
	&\leq C\left[\int_{Q_{\omega_0}}s^{3}\varphi_0^{3}|p_0|^2e^{2s\alpha_0}dxdt+\int_{Q}\left|\sum_{i=1}^n a_{i0}p_i+g_0\right|^2 e^{2s\alpha_0}dxdt\right]\le\\
	&\leq C\left[\int_{Q_{\omega_0}}s^{3}\varphi_0^{3}|p_0|^2e^{2s\alpha_0}dxdt\right.\\
	&\left.+n^2M^2\sum_{i=1}^n \int_{Q}\left|p_i\right|^2 e^{2s\alpha_i}dxdt+\int_{Q}|g_0|^2e^{2s\alpha_i}dxdt\right].
	\end{aligned}
	\end{equation}
\item For $p_i,i\in\overline{1,n}$ we have:
\begin{equation}\label{estCarlemani}\begin{aligned}
&\int_{Q}\left[(s\varphi_i)^{-1}(|D_tp_i|^2+|D^2p_i|^2)+s\varphi_i|Dp_i|^2+s^{3}\varphi_i^{3}|p_i|^2\right]e^{2s\alpha_i}dxdt\\
&\leq C\int_{Q_{\tilde\omega_i}}s^{3}\varphi_i^{3}|p_i|^2e^{2s\alpha_i}dxdt+ C\int_{Q}|g_i|^2e^{2s\alpha_i}dxdt.\\
\end{aligned}
\end{equation}

\end{enumerate}

Summing the above Carleman inequalities we obtain, for some constant $C=C(M,\{\omega_j\}_j)>0$, that 
\begin{equation}
\label{estCarleman1}\begin{aligned}
&\sum_{j=0}^{n}\left\lbrace\int_{Q}\left[(s\varphi_j)^{-1}(|D_tp_j|^2+|D^2p_j|^2)+s\varphi_j|D p_j|^2+s^{3}\varphi_j^{3}|p_j|^2\right]e^{2s\alpha_j}dxdt\right\rbrace\\
&\leq C\left[\int_{Q_{\omega_0}}s^{3}\varphi_0^{3}|p_0|^2e^{2s\alpha_0}dxdt+\sum_{i=1}^{n}\left(\int_{Q_{\tilde\omega_i}}s^{3}\varphi_i^{3}|p_i|^2e^{2s\alpha_i}dxdt\right)\right.\\
&\left.
+\sum_{j=0}^{n}\left(\int_{Q_{\tilde\omega_j}}|g_j|^2e^{2s\alpha_j}dxdt\right)\right].\\
\end{aligned}
\end{equation}

At this point we have to properly estimate the terms containing $p_i$ on $\tilde\omega_i,i\in\overline{1,n}$ from the right hand-side in terms of the  component $p_0$ observed on $\tilde\omega_0$.
For this purpose we will use the first equation of \eqref{adjlinsystem} considered on $\omega_i\cap\omega_0$, which by hypothesis \eqref{hyp}  is coupled only to $p_i$:
\begin{equation}
\label{p0pi}
D_tp_0+\Delta p_0+c_0p_0+a_{i0}p_i=g_0\,\text{ in }(0,T)\times\omega_i\cap\omega_0.
\end{equation}
Consider the  cutoff functions $\gamma_i,i\in\overline{1,n}$ with the properties
\begin{equation*}
\begin{aligned}
&\gamma_i\in C_0^\infty(\omega_i),\,|\gamma_i|\le 1, \text{supp }\gamma_i=\overline{\underline\omega_i} \\
&\gamma_i=\text{ sign }(a_{i_0}|_{\underline\omega_i})\text{ on } \tilde\omega_i,\gamma_i\ne0 \text{ in }\underline\omega_i.
%& |a_{i0}|\ge\gamma_i \cdot a_{i_0}\ge 0 \text{ in }\overline \omega_i,
\end{aligned}
\end{equation*}
where $\text{ sign }(a_{i_0})$ is the sign of $a_{i0}$ in $\underline\omega_i$, which, by hypothesis \eqref{hyp} and continuity is nonzero and constant in $\tilde\omega_i $.
Multiply, scalarly in $L^2(Q_{\omega_0})$, the equation \eqref{p0pi}  by $\gamma_i  s^3\varphi_i^3p_ie^{2s\alpha_i}$:
\begin{equation}\label{ec0multipl}\begin{aligned}
&\int_{Q_{\underline\omega_i}}\gamma_ia_{i0}(x)s^3\varphi_i^3|p_i|^2e^{2s\alpha_i}dxdt\\
&=\int_{Q_{\underline\omega_i}}\gamma_is^3\varphi_i^3(-c_0p_0-D_tp_0-\Delta p_0-g_0)p_ie^{2s\alpha_i}dxdt
\end{aligned}	
	\end{equation}
	We use \eqref{hyp} to say that that there exists a constant such that 
	\begin{equation}\label{ms}
	\begin{aligned}
			\delta	\int_{Q_{\tilde\omega_i}}s^3\varphi_i^3|p_i|^2e^{2s\alpha_i}dxdt&\le \int_{Q_{\tilde\omega_i}}|a_{i0}(x)|s^3\varphi_i^3|p_i|^2e^{2s\alpha_i}dxdt\\
			&\leq \int_{Q_{\underline\omega_i}}a_{i0}(x)s^3\varphi_i^3|p_i|^2e^{2s\alpha_i}dxdt.
\end{aligned}	\end{equation}
	We estimate each term from the right hand-side of \eqref{ec0multipl} using the properties of $\gamma_j, j\in\overline{0,n}$. 
	Let $C>0$ denoting various constants depending on $\delta,M$ and $\underline\omega_i,\tilde\omega_i$. 
	
	For the first term in right side of \eqref{ec0multipl} we have:
	\begin{equation}\label{md1}\begin{aligned}
	&\left| \int_{Q_{\underline\omega_i}}\gamma_is^3\varphi_j^3
	(-c_0p_0)p_ie^{2s\alpha_i}dxdt\right|\\
	&\leq M\left(\int_{Q_{\underline\omega_i}}s^2\varphi_i^2|p_i|^2e^{2s\alpha_i}dxdt\right)^{\frac{1}{2}}\left(\int_{Q_{\underline\omega_i}}s^4\varphi_i^4|p_0|^2e^{2s\alpha_i}dxdt\right)^{\frac{1}{2}}\\
		&\leq
\int_{Q_{\underline\omega_i}}s^2\varphi_i^2|p_i|^2e^{2s\alpha_i}dxdt+M^2\int_{Q_{\underline\omega_i}}s^4\varphi_i^4|p_0|^2e^{2s\alpha_i}dxdt.\end{aligned}
	\end{equation}
	The same computation gives an estimate for the term involving the source:
	\begin{equation}\label{md11}\begin{aligned}
	&\left| \int_{Q_{\underline\omega_i}}\gamma_is^3\varphi_j^3
	(-g_0)p_ie^{2s\alpha_i}dxdt\right|\\
		&\leq
\int_{Q_{\underline\omega_i}}s^2\varphi_i^2|p_i|^2e^{2s\alpha_i}dxdt+M^2\int_{Q_{\underline\omega_i}}s^4\varphi_i^4|g_0|^2e^{2s\alpha_i}dxdt.\end{aligned}
	\end{equation}
	
	Observe now that we have the following estimates for the weight functions, with a constant $cst$ not depending on $s$:
	\begin{equation}
|\gamma_is^3D_t(e^{2s\alpha_i}\varphi_i^3)|=|\gamma_is^3(e^{2s\alpha_i}2sD_t\alpha_{i}\varphi_i^3+3e^{2s\alpha_i}\varphi^2D_t\varphi_i)|\leq
cst\, e^{2s\alpha_i}s^5\varphi_i^5
	\end{equation}
	and
	\begin{equation}
	|s^3\Delta(\gamma_i\varphi_i^3 p_ie^{2s\alpha_i})|\leq cst\,
	s^3\varphi_i^3(s^2\varphi_i^2|p_i|+s\varphi_i|\nabla p_i|+|\Delta
	p_i|)e^{2s\alpha_i}.
	\end{equation}
	
	We now proceed with estimating the second term in \eqref{ec0multipl} using, as usually in Carleman estimates,
 integration by parts:
	\begin{eqnarray}\left|\int_{Q_{\underline\omega_i}}\gamma_i s^3\varphi_i^3(-D_tp_0)p_ie^{2s\alpha_i}dxdt
	 \right|=\left|\int_{Q_{\underline\omega_i}}
	 s^3D_t(\varphi_i^3p_ie^{2s\alpha_i})p_0dxdt \right|\nonumber\\
	\leq\left|\int_{Q_{\underline\omega_i}}s^3D_t(\varphi_i^3e^{2s\alpha_i})p_ip_0dxdt\right|+\left|\int_{Q_{\underline\omega_i}}s^3\varphi_i^3e^{2s\alpha_i}D_tp_ip_0dxdt\right|\nonumber\\
	\leq\label{md2}
	 C\left|\int_{Q_{\underline\omega_i}}e^{2s\alpha_i}s^5\varphi_i^5p_jp_0dxdt\right|+\left|\int_{Q_{\underline\omega_i}}e^{2s\alpha_i}s^3\varphi_i^3D_tp_ip_0dxdt\right|\nonumber\\
	\leq\int_{Q_{\underline\omega_i}}s^2\varphi_i^2|p_i|^2e^{2s\alpha_i}dxdt+C\int_{Q_{\underline\omega_i}}s^8\varphi_i^8|p_0|^2e^{2s\alpha_i}dxdt
	\\
	 +\int_{Q_{\underline\omega_i}}(s\varphi)^{-2}|D_tp_i|^2e^{2s\alpha_i}dxdt+C\int_{Q_{\underline\omega_i}}s^8\varphi_i^8|p_0|^2e^{2s\alpha_i}dxdt.\nonumber
	\end{eqnarray}
	We proceed now with estimating the third term in right hand side of \eqref{ec0multipl}:
	\begin{equation}\label{md3}
	 \begin{aligned}
	 &\left|\int_{Q_{\underline\omega_i}}\gamma_is^3\varphi_i^3(-\Delta
	 p_0)p_ie^{2s\alpha_i}dxdt\right|=\left|
	 \int_{Q_{\underline\omega_i}}s^3\Delta(\gamma_i\varphi_i^3 p_i
	 e^{2s\alpha_i})p_0dxdt\right|\\
	&\leq
	 C\int_{Q_{\underline\omega_i}}s^3\varphi_i^3(s^2\varphi_i^2|p_i|+s\varphi_i|\nabla
	 p_i|+|\Delta p_i|)e^{2s\alpha_i}|p_0|dxdt\\  
	 &\leq
	 \int_{Q_{\underline\omega_i}}[s^2\varphi_i^2|p_i|^2+|\nabla
	 p_i|^2+(s\varphi_i)^{-2}|\Delta p_i|^2]e^{2s\alpha_i}dxdt\\
	 &+C\int_{Q_{\underline\omega_i}}s^8\varphi_i^8
	 |p_0|^2e^{2s\alpha_i}dxdt.
	 \end{aligned}
\end{equation}
	
Using \eqref{md1},\eqref{md11},\eqref{md2}, \eqref{md3} and \eqref{ms} we have, for $i\in\overline{1,n}$ that
\begin{equation}\begin{aligned}
	&\int_{Q_{\tilde\omega_i}}s^3\varphi_i^3|p_i|^2e^{2s\alpha_i}dxdt\le C \int_{Q_{\underline\omega_i}}s^8\varphi_i^8
	 |p_0|^2e^{2s\alpha_i}dxdt\\
	 &+\int_{Q_{\underline\omega_i}}\left[(s\varphi_i)^{-2}(|\Delta p_i|^2+|D_tp_i|^2)+s^2\varphi_i^2|p_i|+|\nabla
	 p_i|^2\right]e^{2s\alpha_i}dxdt\\
	 &+C\sum_{i=1}^n\int_{Q_{\underline\omega_i}}s^4\varphi_0^4
			|g_0|^2e^{2s\alpha_i}dxdt.\\
\end{aligned}\end{equation}

Going back to \eqref{estCarleman1}, we have

\begin{equation}
\begin{aligned}
&\sum_{j=0}^{n}\left\lbrace\int_{Q}\left[(s\varphi_j)^{-1}(|D_tp_j|^2+|D^2p_j|^2)+s\varphi_j|D p_j|^2+s^{3}\varphi_j^{3}|p_j|^2\right]e^{2s\alpha_j}dxdt\right\rbrace\\
&\leq C\int_{Q_{\omega_0}}s^3\varphi_0^3
|p_0|^2e^{2s\alpha_0}dxdt+C\sum_{i=1}^{n}\left(\int_{Q_{\underline\omega_i}}s^8\varphi_i^8
	 |p_0|^2e^{2s\alpha_i}dxdt\right.\\
	 &
\left.+\int_{Q_{\underline\omega_i}}\left[(s\varphi_i)^{-2}(|D^2 p_i|^2+|D_tp_i|^2)+s^2\varphi_i^2|p_i|^2+|D
	 p_i|^2\right]e^{2s\alpha_i}dxdt\right)\\
	 &+C\sum_{i=1}^n\int_{Q_{\underline\omega_i}}s^4\varphi_0^4
			|g_0|^2e^{2s\alpha_i}dxdt+ C\sum_{j=0}^{n}\int_{Q}
			|g_j|^2e^{2s\alpha_j}dxdt.\\
\end{aligned}
\end{equation}
	 We now absorb the integral terms containing $p_i$ in the right hand side into the corresponding higher order terms in the left side of the above inequality, by increasing $s$ and taking it big enough. We obtain: 
	\begin{equation}
		\label{estCarleman2}\begin{aligned}
			&\sum_{j=0}^{n}\left\lbrace\int_{Q}\left[(s\varphi_j)^{-1}(|D_tp_j|^2+|D^2p_j|^2)+s\varphi_j|D p_j|^2+s^{3}\varphi_j^{3}|p_j|^2\right]e^{2s\alpha_j}dxdt\right\rbrace\\
			&\leq C\int_{Q_{\omega_0}}s^3\varphi_0^3
			|p_0|^2e^{2s\alpha_0}dxdt+ C\sum_{i=1}^{n}\int_{Q_{\underline\omega_i}}s^8\varphi_i^8
			|p_0|^2e^{2s\alpha_i}dxdt.\\
			&+C\sum_{i=1}^n\int_{Q_{\underline\omega_i}}s^4\varphi_0^4
			|g_0|^2e^{2s\alpha_i}dxdt+ C\sum_{j=0}^{n}\int_{Q}
			|g_j|^2e^{2s\alpha_j}dxdt.\\
		\end{aligned}
	\end{equation} 
	Now  we use Remark \ref{ordineaponderilor} in order to take a smaller weight in the left side and a greater one in the right side. Then  there exist $s_0>0$ and $C=C(M,\delta,\{\underline\omega_j\}_j)$ such that    the following Carleman estimate is true for all $s\ge s_0$: 
	\begin{equation}
		\label{estCarleman3}\begin{aligned}
		&\sum_{j=0}^{n}\left[\int_{Q}\left(|D_tp_j|^2+|D^2p_j|^2+|D p_j|^2+|p_j|^2\right)e^{2s\underline\alpha}dxdt\right]\\
		&\leq C\int_{Q_{\omega_0}}|p_0|^2e^{2s\overline\alpha}dxdt+C\int_{Q}|g|^2e^{2s\overline\alpha}dxdt.\\
		\end{aligned}
	\end{equation}
	
\end{proof}

Concerning the $L^\infty-L^2$ Carleman estimate for the solution of the adjoint problem \eqref{adjlinsystem} we proceed in the same way as is in \cite{bar12002,corgueros2010} or \cite{balch}. 
We need to use the maximal regularity result in $L^p$ spaces for parabaolic problems (see \cite{lady}) and   Sobolev embeddings for anisotropic Sobolev spaces which are contained in the following lemma:

\begin{lemma}[\cite{lady}, Lemma 3.3]\label{lemmaLady1}
	Let $z\in W^{2,1}_r(Q)$.
	
	Then $z\in Z_1$ where 
	$$
	Z_1= \left\{\begin{array}{lll}
	L^s(Q)&\text{ with }s\le \frac{(N+2)r}{N+2-2r},&\text{ when }r<\frac{N+2}{2},\\
	L^s(Q)&\text{ with }s\in[1,\infty),&\text{ when }r=\frac{N+2}{2},\\
	C^{\alpha,\alpha/2}(Q)& \text{ with } 0<\alpha <2-\frac{N+2}{r},&\text{ when }r>\frac{N+2}{2},
	\end{array}\right.
	$$	
	and there exists $C=C(Q,p,N)$ such that $$
	\|z\|_{Z_1}\le C\|z\|_{W^{2,1}_r(Q)}.
	$$
\end{lemma}
 Using the above regularity result we consider the following sequence of numbers: 
\begin{equation}\label{seqsigmaj}
\sigma_0=2,\quad \sigma_j:=
\begin{cases}
\dfrac{(N+2)\sigma_{j-1}}{N+2-2\sigma_{j-1}}, \text{  if } \sigma_{j-1}<\frac{N+2}{2},\\
\frac{3}{2}\sigma_{j-1}, \text{  if } \sigma_{j-1}\geq \frac{N+2}{2},
\end{cases}
\end{equation}
such that by Lemma \ref{lemmaLady1} we have 
$$W^{2,1}_{\sigma_{m-1}}(Q)\subset L^{\sigma_m}(Q).$$
Now, let us fix a $\delta_1>0$ and  a sequence $(q^j)_{j>0}$ defined by
\begin{equation*}
q^j:=p^\varepsilon e^{(s+j\delta_1)\underline{\alpha}}.
\end{equation*} 
Then $q^j=(q^j_0,\ldots,q^j_n)^\top$ is  solution to the problem
\begin{equation}\begin{aligned}
&D_tq^j+\textbf{A}q^j+C q^j+A_0^\top q^j=(s+j\delta_1)D_t\underline\alpha  q^j,\\
&q^j(T)=0.\\
\end{aligned}\end{equation}

%\begin{equation}\label{sysqj}
%\left\{\begin{aligned}
%&-D_tq_0^j-\Delta q_0^j-(s+j\delta_1)D_t\overline\alpha q_0^j=c_0q_0^j+\sum_{i=1}^na_{i0}q_0^j,\text{ in }(0,T)\times\Omega\\
%&-D_tq_i^j-\Delta q_i^j-(s+j\delta_1)D_t\overline\alpha q_0^i=c_iq_i^j, \text{ in }(0,T)\times\Omega\\
%&q_0^j=...=q_n^j=0,\text{ on } (0,T)\times\partial\Omega,\\
%&q^j_0(0,\cdot)=...=q^j_i(0,\cdot)=0 \text{ in }\Omega
%\end{aligned}\right.\quad
%j=\overline{1,m}.
%\end{equation} 

Observe that   the right-hand side may be bounded in terms of $q^{j-1}$, with some constant $C_j=C_j(s,\delta_1)>0$, as follows 
\begin{equation}
(s+j\delta_1)D_t\underline\alpha  q^j=(s+j\delta_1)\frac{2t-T}{t(T-t)}\underline{\alpha}e^{\delta_1\underline{\alpha}}q^{j-1}\leq C_jq^{j-1}.
\end{equation}
By  maximal parabolic regularity (see \cite{lady}) we have 
\begin{equation}\label{parabregj}
\|q^j\|_{W^{2,1}_{\sigma_{j-1}}}\leq \tilde C_j\|q^{j-1}\|_{L^{\sigma_{j-1}}}
\end{equation}
and using Sobolev type embedding from  Lemma \ref{lemmaLady1}, we have that there exists a constant $K_j$ such that
\begin{equation}\label{soboj}
\|q^{j-1}\|_{L^{\sigma_{j-1}}}\leq K_j \|q^{j-1}\|_{W^{2,1}_{\sigma_{j-2}}}.
\end{equation}
The sequence $(\sigma_m)_m$ is increasing to $+\infty$ and choose rank $m_0$ such that $\sigma_{m_0}>\frac{N+2}{2}\ge\sigma_{m_0-1}$. This implies that
\begin{equation}\label{inftym}
W^{2,1}_{\sigma_{m_0}}(Q)\subset L^{\infty}(Q).
\end{equation}
From  \eqref{parabregj}, \eqref{soboj} and \eqref{inftym}, and with the use of \eqref{estCarlemanstar}, we have that there exists a constant $C>0$ such that 
\begin{equation}\begin{aligned}
&\|p e^{(s+m_0\delta_1)\underline{\alpha}}\|_{L^{\infty}(Q)}=\|q^{m_0}\|_{L^\infty(Q)}\leq C\|q^0\|_{L^{\sigma_0}(Q)}=C\|p e^{s\underline{\alpha}}\|_{L^2(Q)}\\
&\leq C\|p_0e^{s\overline{\alpha}}\|_{L^2(Q_{\omega_0})} .
\end{aligned}\end{equation}

\begin{remark}\label{rem_obs}
In order to obtain the observability inequality we proceed in the classical manner, by multiplying scalarly in $L^2(\Omega)$ each equation  of  the system \eqref{genericparab} by $p_i$ and making use of dissipativity to find, for some constant $c>0$ depending only on the coefficients of the system, the inequality:
	%\begin{equation*}
	%\frac{1}{2}\frac{d}{dt}\|p\|^2_{L^2(\Omega)}-\|\nabla p\|^2_{L^2(\Omega)}+\langle c(x)p,p\rangle_{L^2(\Omega)}=0
	%\end{equation*}
	
	\begin{equation*}
	\frac{1}{2}\frac{d}{dt}\|p\|^2_{L^2(\Omega)}+c \|p\|^2_{L^2(\Omega)}\geq 0,
	\end{equation*} 
	which gives
	\begin{equation*}
	\|p(0)\|^2_{L^2(\Omega)}\leq \|p(t)\|^2_{L^2(\Omega)}e^{Ct}, t\in(0,T).
	\end{equation*}
	Consequently, for fixed  $s>s_0$, we have that 
	\begin{equation*}
	\|p(0)\|^2_{L^2(\Omega)}\leq \frac{T}{2}\int_{\frac{T}{4}}^{\frac{3T}{4}} \|p(t)\|^2_{L^2(\Omega)}e^{Ct}dt\leq K(T,s)\int_{0}^{T} \|p(t)\|^2_{L^2(\Omega)}e^{2s\underline\alpha}dt.
	\end{equation*} 
	Now, by Carleman estimate \eqref{estCarleman3} we obtain the observability inequality:
	\begin{equation}\label{obsineq}
	\|p(0,\cdot)\|^2_{L^2(\Omega)}\leq C\int_{Q_{\omega_0}}
	|p_0|^2e^{2s\overline\alpha}dxdt,
	\end{equation}
	with a constant $C=C(T,s,\delta,M,\{\underline\omega_j\}_j)$.
	
\end{remark}

\section{Linear system: null controllability}\label{seccontrtreelin}

The main controllability result concerning linear system \eqref{linsystem} is the following
\begin{theorem}\label{thcontrol}
	Consider system \eqref{linsystem} with coefficients in $\mathcal{E}_{M,\delta,\{\underline\omega_i\}_i}$. Then there exists a constant $C=C({M,\delta,\{\underline\omega_i\}_i})$ such that for all  $z^0\in  H$ there exists $u^*\in L^2(0,T;L^2(\omega_0))\cap L^\infty(Q_{\omega_0})$ which drives the corresponding solution to   \eqref{linsystem}, $z=z^{u^*}$ in $0$ \textit{i.e.}   $z(T,\cdot)=0$ and  satisfies the norm estimate 
	\begin{equation}\label{est_contr_l2}
	\|u^*e^{-s\overline\alpha}\|_{L^2(0,T;L^2(\omega_0))}+	\|u^* \|_{L^\infty(Q_{\omega_0})}\leq C\|z^0\|_{L^2(\Omega)}.
	\end{equation}
\end{theorem}
\newpage

\proof\quad

\subsubsection*{$L^2(Q)$ control.}
In order to obtain norm estimates for the controls driving the trajectory to the linear system in $0$, we consider a family of optimal control problems depending on a small parameter  $\varepsilon>0$:    
\begin{equation}\label{optimalpb}
\inf_{u\in L^2(Q_{\omega_0})}\frac{1}{2}\int_{Q_{\omega_0}}|u|^2e^{-2s\overline\alpha}dxdt+\frac{1}{2\varepsilon}\int_{\Omega}|z(T,\cdot)|^2dxdt, 
\end{equation} 
with $z=z^u$ the solution of the linear controlled system \eqref{linevpb}.
Classical results concerning optimal control with quadratic cost for parabolic equations  insure existence of optimal control $u^\varepsilon$ which by Pontriaghin maximum principle satisfy 
\begin{equation}\label{optimalcontrol}
u^\varepsilon=e^{2s\overline\alpha}\textbf{B}^*p^\varepsilon=e^{2s\overline\alpha}p_0^\varepsilon|_{\omega_0}.
\end{equation}
where $p^\varepsilon$ is solution to the adjoint system:
\begin{equation}
\begin{cases}
D_tp^\varepsilon=-\textbf{A}p^\varepsilon-\textbf{C}(t) p^\varepsilon-\textbf{A}_0^*(t) p^\varepsilon,\\
p^\varepsilon(T)=-\frac{1}{\varepsilon}z^\varepsilon(T).
\end{cases}
\end{equation}
By cross multiplying the equations for $z^\varepsilon=z^{u^\varepsilon}$ and $p^\varepsilon$ by $p^\varepsilon$ respectively $z^\varepsilon$ and integrating on $Q$ we obtain:
$$\frac{d}{dt}\langle z^\varepsilon,p^\varepsilon\rangle_{L^2(\Omega)}=\langle (A+A_0+C)z^\varepsilon+Bu^\varepsilon,p^\varepsilon\rangle_{L^2(\Omega)}-\langle (A+A_0+C)^*p^\varepsilon,z^\varepsilon\rangle_{L^2(\Omega)}.$$
We integrate  on  $[0,T]$ and use the observability inequality \eqref{obsineq} to get %, with $s_1>0$ chosen such that $$s_1>s \text{ and } s_1\overline{\alpha}<s\underline{\alpha}<s\overline{\alpha}.$$ Then, with  a  constant $C=C(T,s,s_1,\delta,M,\{\underline\omega_j\}_j)$ we have
\begin{equation*}\begin{aligned}
&\frac{1}{\varepsilon}\|z^\varepsilon(T,\cdot)\|_{L^2(\Omega)}^2+\langle u^\varepsilon,B^*p^\varepsilon\rangle_{L^2(Q)}=-\langle z^\varepsilon(0,\cdot),p^\varepsilon(0,\cdot)\rangle_{L^2(\Omega)}\\
&\leq \|z^0\|_{L^2(\Omega)}\|p(0,\cdot)\|_{L^2(\Omega)}\leq C\|z^0\|_{L^2(\Omega)}\left(\int_{Q_{\omega_0}}
|p^\varepsilon_0|^2e^{2s\overline\alpha}dxdt\right)^{\frac{1}{2}}.
\end{aligned}
\end{equation*}

Since $\langle u^\varepsilon,B^*p^\varepsilon\rangle_{L^2(Q)}=\int_{Q_{\omega_0}}
		|p^\varepsilon_0|^2e^{2s\overline\alpha}dxdt$, using appropriately balanced  Young's inequality, we find that
\begin{equation}
\frac{1}{\varepsilon}\|z^\varepsilon(T,\cdot)\|_{L^2(\Omega)}^2+\frac{1}{2}\int_{Q_{\omega_0}}|p^\varepsilon_0|^2e^{2s\overline\alpha}dxdt\leq C\|z^0\|_{L^2(\Omega)}^2,
\end{equation}
and gives by \eqref{optimalcontrol} the following  estimate for the sequence of optimal controls $(u^\varepsilon)_\varepsilon$ and final state:
\begin{equation}
\frac{1}{\varepsilon}\|z^\varepsilon(T,\cdot)\|_{L^2(\Omega)}^2+\frac{1}{2}\int_{Q_{\omega_0}}
|u^\varepsilon|^2e^{-2s\overline\alpha}dxdt\leq C\|z^0\|_{L^2(\Omega)}^2.
\end{equation}

Now, this $L^2$ bound for the sequence $(u^\varepsilon)_\varepsilon$, allows to extract a subsequence, denoted for simplicity also $(u^\varepsilon)_\varepsilon$ weakly convergent in $L^2(Q)$ to a limit $u^*$.

Write the  corresponding solutions $(z^\varepsilon)_\varepsilon$ as 
$$z^\varepsilon=w^\varepsilon+v
$$
where $w^\varepsilon$ is solution to \eqref{linsystem} with initial data $w^\varepsilon(0)=0$ and $v$ solution to  homogeneous equation
$$
D_tv=\textbf{A}v+(A_0+C)v=0,\,v(0)=z^\varepsilon(0)=z^0.
$$
We have that the sequence $(w^\varepsilon)_\varepsilon$ is bounded in $L^2(0,T; D(\textbf{A}))$  and the sequence of derivatives $({D_tw^{\varepsilon}})_\varepsilon$ is bounded in $L^2(0,T;L^2(\Omega))$. By Aubin's theorem we can extract a subsequence, denoted also $(w^\varepsilon)_\varepsilon$, strongly convergent in  $L^2(0,T; H_0^1(\Omega))$ to $w\in L^2(0,T; H_0^1(\Omega))\cap L^2(0,T; D(\textbf{A}))$. Consequently $(z^\varepsilon)$ is strongly convergent in  $L^2(0,T; H_0^1(\Omega))$ to $z\in L^2(0,T; H_0^1(\Omega)).$
We may now pass to the limit in the weak formulation of solutions to \eqref{linsystem}, \eqref{linevpb}; thus, for some test function $\bm\varphi\in [H^1_0(\Omega)]^{n+1}$,  we have
\begin{equation}
\begin{cases}
\displaystyle\langle z^\varepsilon(t,\cdot),\bm\varphi\rangle_{L^2(\Omega)}-\langle z^\varepsilon(0,\cdot),\bm\varphi\rangle_{L^2(\Omega)}+\int_0^t\langle \nabla z^\varepsilon(\tau,\cdot),\nabla\bm\varphi \rangle_{L^2(\Omega)}d\tau\\ \\\displaystyle
+\int_0^t\langle(A_0+C) z^\varepsilon,\bm\varphi\rangle_{L^2(\Omega)} d\tau=\int_{(0,t)\times\omega_0}u^\varepsilon\bm\varphi dxd\tau,\\
z^\varepsilon(0,\cdot):=z^0,
\end{cases}
\end{equation}
and we find that  $z\in L^2(Q)$ is solution to the problem  (\ref{linevpb}) with initial datum $z^0\in L^2(\Omega)$. In fact, by Arzel\`a-Ascoli theorem $w^\varepsilon\rightarrow w$ in $C([0,T],L^2(\Omega))$ and thus   $z(T)=0$ and by weak lower semicontinuity of the $L^2$ norm we also have the following estimate for the control driving the solution to 0:
\begin{equation}\label{boundL2control}
\int_{Q_{\omega_0}}|u^*|^2 e^{-2s\overline\alpha}dxdt\leq C\|z^0\|_{L^2(\Omega)}^2.
\end{equation}
where $C=C(T,s,s_1,M,\delta,\{\underline\omega_j\}_j).$

\subsubsection*{$L^\infty(Q)$- control.} 

Regarding the  $L^\infty$ norm estimates for the sequence $(u^\varepsilon)_{\varepsilon}$ and also for $u^*$  we will use the results from the previous section \S\ref{secobservcarl}:
\begin{equation}\label{aaa}
\|u^\varepsilon e^{-2s\overline\alpha+(s+m_0\delta_1)\underline{\alpha}}\|_{L^{\infty}(Q_{\omega_0})}=\|p_0^\varepsilon e^{(s+m_0\delta_1)\underline{\alpha}}\|_{L^{\infty}(Q_{\omega_0})}\leq C \|z^0\|_{L^2(Q)}.
\end{equation}
Now we see that we could start from the beginning with $\lambda$ big enough such that  \eqref{ordineaponderilor2} holds and in consequence
  $$2s\overline\alpha\leq(s+m_0\delta_1)\underline{\alpha}.$$
 As $-2s\overline\alpha+(s+m_0\delta_1)\underline{\alpha}>0$, by passing to $L^\infty$ weak-* limit in \eqref{aaa}, we find that 
  \begin{equation}
  \|u^*\|_{L^{\infty}(Q_{\omega_0})}\leq\|u^* e^{-2s\overline\alpha+(s+m_0\delta_1)\underline{\alpha}}\|_{L^{\infty}(Q_{\omega_0})}\leq C \|z^0\|_{L^2(Q)},
  \end{equation}
which concludes \eqref{est_contr_l2}.

\fin
%%%%%%%%%%%%%%%%%%%%%%%%%%%%%%%%%%%%%%%%%%%%%%%%%%%%%%%

\section{Nonlinear system: local exact controllability}

We prove in this section the following local controllability result concerning system \eqref{nonlinsystem}:
\begin{theorem}\label{th_local_contr}
	Suppose $\overline y$ is a stationary state, \textit{i.e.} solution to \eqref{nonlinelliptic}, and that the functions $f_j,j\in\overline{0,n}$ satisfy hypotheses \textit{(H1)}, \textit{(H2)}. Then, for all $\beta_0>0$ there exist  $\zeta_0=\zeta_0(\beta_0)>0$ and $C=C(\beta_0,\{\underline \omega_i\}_i,\overline{y})$ such that if $\|y^u(0)-\overline y\|<\zeta_0$ there exists a control $u\in L^\infty(Q)$ satisfying
$$
\|u\|_{L^\infty(Q)}\le C\|y^u(0)-\overline y\|_{L^\infty(\Omega)}
$$
and 
	$$y^u(T,\cdot)=\overline y,$$
	with
	$$ \|y(t,\cdot)-\overline y\|_{L^\infty}\le \beta_0, \,t\in[0,T].$$
\end{theorem}

\label{seccontrtreenonlin}

\proof
The approach to the  local null controllability of the system around the stationary state is based on the  Kakutani fixed point theorem.

For this aim,  with a solution $y$ to \eqref{nonlinsystem}, we consider the system satisfied by $z:=y-\overline{y}$, written as a linear system 
\begin{equation}\label{linsys2}
\left\lbrace
\begin{array}{ll}
D_tz_{0}-\Delta z_0=c_0^z(t,x)z_0+\chi_{\omega_0}u, & (0,T)\times\Omega,\\
D_tz_{i}-\Delta z_i=a_{i0}^z(t,x)z_0+c_i^z(t,x)z_i,\, i\in\overline{1,n},\, &(0,T)\times\Omega,\\
z_0=...=z_{n}=0, &(0,T)\times\partial\Omega,\\
z(0,x)=z^0(x):=y(0,x)-\overline y(x)&x\in\Omega,
\end{array}
\right.
\end{equation} 
where the nonlinearity is hidden into the coupling coefficients which are defined by:

\begin{eqnarray}\label{linearizationcoef}
\nonumber a^z_{i0}(t,x):=\int_0^1\frac{\partial}{\partial y_0}f_i(x,\overline y_0(x)+\tau z_0(t,x),\overline y_i(x)+\tau z_i(t,x))d\tau, \, i\in\overline{1,n}\\
\nonumber c^z_{j}(t,x):=\int_0^1\frac{\partial}{\partial y_j}f_j(x,\overline y_0(x)+\tau z_0(t,x),\overline y_j(x)+\tau z_j(t,x))d\tau, \,  j\in\overline{0,n}.\\
\end{eqnarray}

Observe that $\{a^0_{i0}, c^0_{j}\}_{i\in\overline{1,n},j\in\overline{0,n}}$ are the coefficients of the linearized system around the stationary solution $\overline y$ as
$$
a^0_{i0}(x)=\frac{\partial}{\partial y_0}f_i(x,\overline y_0(x) ,\overline y_i(x) ),$$
$$ c^0_{i}= \frac{\partial}{\partial y_i}f_i(x,\overline y_0(x), \overline y_i(x) ),c^0_{0}= \frac{\partial}{\partial y_0}f_0(x,\overline y_0(x) ) .
$$
We see now that hypotheses \eqref{fsuport} and \eqref{suportderiv} tell us that we may choose $M_0,\delta_0>0$ and $\underline\omega_i\subset\subset (\omega_i\cap\omega_0)\setminus\bigcup_{j\ne0,i}\omega_j$ such that 
\begin{equation}\label{coef_lin_E}
\{a^0_{i0}, c^0_{j}\}_{i\in\overline{1,n},j\in\overline{0,n}}\in\mathcal{E}_{M_0,\delta_0,\{\underline\omega_i\}_i}.
\end{equation}
Let  $\beta>0$ and define $\mathcal{M}_\beta$ to be:
\begin{equation}
\mathcal{M}_\beta=\lbrace \tilde{z}\in L^\infty(Q):\|\tilde{z}\|_{L^\infty(Q)}\leq\beta\rbrace.
\end{equation}
For  $\tilde{z}\in\mathcal{M}_\beta$, we consider the coefficients $a^{\tilde z}_{i0}(x),\, c^{\tilde z}_{j}(x) $ defined as in 
\eqref{linearizationcoef} with $z$ replaced by $\tilde z$. 

Observe now that  we may choose $\beta_0>0$ small enough such that if $\tilde z\in \mathcal{M}_{\beta_0}$ we have 
\begin{equation}
\label{coef_z_tilde}
\{a^{\tilde z}_{i0}, c^{\tilde z}_{j}\}_{i\in\overline{1,n},j\in\overline{0,n}}\in\mathcal{E}_{2M_0,\frac{\delta_0}{2},\{\underline\omega_i\}_i}.
\end{equation} 
Consider now the linear system  \eqref{linsys2} with coefficients $\{a^{\tilde z}_{i0}, c^{\tilde z}_{j}\}$:
\begin{equation}\label{linsys2.1}
\left\lbrace
\begin{array}{ll}
D_tz_{0}-\Delta z_0=c_0^{\tilde z}(t,x)z_0+\chi_{\omega_0}u, & (0,T)\times\Omega,\\
D_tz_{i}-\Delta z_i=a_{i0}^{\tilde z}(t,x)z_0+c_i^{\tilde z}(t,x)z_i,\, i\in\overline{1,n},\, &(0,T)\times\Omega,\\
z_0=...=z_{n}=0, &(0,T)\times\partial\Omega,\\
z(0,x)=z^0(x)&x\in\Omega.\end{array}
\right.
\end{equation} 
The linear   problem \eqref{linsys2.1} may be reformulated as: 
\begin{equation}\label{linsys2.2}
\left\lbrace
\begin{array}{ll}
D_tz=\textbf{A}z+\mathbf{A_0^{\tilde z}}(t)z+\mathbf{C^{\tilde z}}(t)z+\textbf{B}u, & t>0, \\
z(0)=z^0,   &\\
\end{array}
\right.
\end{equation} 
where $\mathbf{C^{\tilde z}}(t)z=C^{\tilde z}_0(t,\cdot)z(\cdot)$ and $\mathbf{A^{\tilde z}_0}(t)z=A^{\tilde z}_0(t,\cdot)z(\cdot)$ where  $C^{\tilde z}_0(t,x)=diag(c^{\tilde z}_i(t,x))_{i=\overline{0,n}}$ and the coupling matrix  $$A^{\tilde z}_0(t,x)=(0,a^{\tilde z}_{10}(t,x),\ldots,a^{\tilde z}_{n0}(t,x))^\top\cdot (1,0,\ldots,0).$$
\smallskip

 Theorem \ref{thcontrol} says that for $\tilde z\in \mathcal{M}_{\beta_0}$ there exists a control $u^*=u^*(\tilde z)\in L^2(0,T;L^2(\omega_0))\cap L^\infty(Q_{\omega_0})$   satisfying the norm estimate 
 \begin{equation}\label{est_contr_l2_0}\begin{aligned}
 & J(u^*):=\|u^*e^{-s\overline\alpha}\|_{L^2(0,T;L^2(\omega_0))}+	\|u^* \|_{L^\infty(Q_{\omega_0})}\\
 &  \leq C(2M_0,\delta_0/2,\{\underline\omega_i\}_i)\|z^0\|_{L^2(\Omega)},
 \end{aligned}
 \end{equation}
 and driving the  solution $z^{u^*,\tilde z}$ of the linear system  \eqref{linsys2.1} in zero : $z^{u^*,\tilde z}(T)=0$.
 Observe that $J$ is a norm in the space $\mathcal{U}^*:=L^2_{e^{-s\overline\alpha}}\cap L^\infty (Q_{\omega_0})$.
 
 We will write \begin{equation}
 \label{T1T2}
 z^{u,\tilde z}=T_1^{\tilde z}(z^0)+T_2^{\tilde z}(u),
 \end{equation}
 where the first term is the solution to problem \eqref{linsys2.1} with initial data $z^0$ and the second term is the solution to system \eqref{linsys2.1} with initial datum zero and control $u$. 
 Let us denote by 
 \begin{equation}
 \label{op_S}
 S_1(z^0)=e^{t\textbf{A}}z^0,\, S_2h=e^{t\textbf{A}}* h=\int_0^te^{(t-s)\textbf{A}}h(s)ds,
 \end{equation}
 where $h\in L^2(0,t;[L^2(\Omega)]^{n+1})$.
 With these notations
\begin{equation} \label{S-T} 
 z^{u,\tilde z}=T_1^{\tilde z}(z^0)+T_2^{\tilde z}(u)=S_1(z^0)+S_2(A_0^{\tilde z}z^{u,\tilde z}+C_0^{\tilde z}z^{u,\tilde z}+Bu).
\end{equation}
 
 \medskip
 
Fix an initial datum $z^0\in L^\infty(\Omega)$. We define now the following set-valued map, associated to $z^0$: 
\begin{equation}
\begin{aligned}
&F_{z^0}:\mathcal{ M}_{\beta_0}\rightarrow 2^{L^\infty(Q)}\\
F_{z^0}(\tilde z)&= \{z^{u,\tilde z}: u \text{ satisfies } \eqref{est_contr_l2_0} \text{ and } z^{u,\tilde z} (T)=0\}\\
&=\{T^{\tilde z}_1(z^0)+T^{\tilde z}_2(u): \, z^{u,\tilde z} (T)=0, J(u)\le K\|z^0\|_{L^2}\},
\end{aligned}
\end{equation}
where by $K$ we denoted the constant in \eqref{est_contr_l2_0}, $K=C(2M_0,\delta_0/2,\{\underline\omega_i\}_i)$.

In order to obtain local controllability of the nonlinear system it is enough to find a fixed point for $F_{z^0}$. We achieve this goal by applying Kakutani fixed point theorem to $F_{z^0}$ in $\mathcal{ M}_{\beta_0}$; we have thus to verify the following statements:
\begin{enumerate}
	\item[i)] For every $\widetilde{z}\in \mathcal M$, $F_{z^0}(\widetilde{z})$ is a nonempty, closed and convex subset of $L^\infty(Q)$;
	\medskip
	
	Observe that $z^{u^*(\tilde z)}\in F_{z^0}(\widetilde{z})$ and thus   $F_{z^0}(\widetilde{z})\not=\emptyset$.  Convexity comes from linearity of $T_2$ and convexity of $J$.
	
	To prove that  $F_{z^0}(\widetilde{z})$ is closed,  suppose $z^m\in F_{z^0}(\widetilde{z})$, $z^m\rightarrow z$ in $L^\infty$. We have to prove that $z\in  F_{z^0}(\widetilde{z})$. Indeed, we have that 
	$$
	z^m=T_1^{\tilde z}(z^0)+T_2^{\tilde z}(u^m)
	$$
	for some controls $u^m\in\mathcal{U}^*$ satisfying estimate $J(u^m)\le K\|z^0\|_{L^2}$.   
	We may now invoke Aubin-Lions and  Ascoli-Arzel\`a compactness results  (see \textit{e.g.} \cite{vra_c0}) applied to the solution operator of a parabolic initial boundary value  problem and thus to say that 
	$T_2$ is a compact operator from $L^2(0,T;L^2(\Omega_{\omega_0}))$ to $C([0,T];[L^2(\Omega)]^{n+1})\cap L^2(0,T;[H_0^1(\Omega)]^{n+1})$. Thus, extracting subsequence $u^m\rightharpoonup u$ weakly in $L^2(Q_{\omega_0})$ we find $$z^m\rightarrow z \text{ in } C([0,T];[L^2(\Omega)]^{n+1})\cap L^2(0,T;[H_0^1(\Omega)]^{n+1})$$ with $z(T)=0$ since $z_m(T)=0$. Thus $z\in  F_{z^0}(\widetilde{z}) $.
	\item[ii)] There exists $\zeta_0=\zeta_0(\beta_0)$ such that for $\|z^0\|_{L^\infty(\Omega)}<\zeta_0$  we have 
	$$F_{z^0}(\mathcal{ M}_{\beta_0})\subset \mathcal{ M}_{\beta_0}.$$
 This follows from the a priori estimates for solutions to initial boundary value problems for parabolic systems:
	$$
	\|T_1^{\tilde z}(z^0)\|_{L^\infty(\Omega)}\le C_1(\|\tilde z\|_{L^\infty})\|z^0\|_{L^\infty(\Omega)},
	$$
	$$
	\|T_2^{\tilde z}(u)\|_{L^\infty(\Omega)}\le C_2(\|\tilde z\|_{L^\infty})\|u\|_{L^\infty(Q_{\omega_0})}
	$$
	and from the remark that both constants depend in fact uniformly on the $L^\infty$ norm of the coupling coefficients and thus depend uniformly on the norm of $\tilde z$ in $L^\infty$.
	 
	\item[iii)]The set $F_{z^0}(\mathcal{ M}_{\beta_0})$ is imbedded into a convex and compact subset of $\mathcal{ M}_{\beta_0}$.
	
	Indeed, as $\mathcal{ M}_{\beta_0}$ is closed and convex, it is enough to prove that $F_{z^0}(\mathcal{ M}_{\beta_0})$ is relatively compact in $L^\infty$ topology. For this, take a sequence $z^m\in F_{z^0}(\mathcal{ M}_{\beta_0}) $. Correspondingly, there exist $\tilde z^m\in \mathcal{ M}_{\beta_0}$ with $z^m\in F_{z^0}(\tilde z^m)$. Take corresponding controls $u^m\in\mathcal{U}^*$ such that (see definition of $F_{z^0}$ and \eqref{S-T}):
	\begin{equation}\label{eqlim}
    z^m=T_1^{\tilde z^m}(z^0)+T_2^{\tilde z^m}(u^m)=  S_1(z^0)+S_2(A_0^{\tilde z^m}z^m+C_0^{\tilde z^m}z^m+Bu^m).
    \end{equation}
	We have the following bounded sequences
	\begin{itemize}
		\item  $\tilde z^m\in \mathcal{M}_{\beta_0}$  and so $A_0^{\tilde z^m}(Q),C_0^{\tilde z^m}(Q)$  are bounded in $L^\infty$;
		\item $z^m\in \mathcal{M}_{\beta_0}$ and is thus bounded in $L^\infty(Q)$;
		\item $u^m\in\mathcal{U}^*$ is bounded in $L^\infty(Q)$.
	\end{itemize}
Consequently  $ A_0^{\tilde z^m}z^m+C_0^{\tilde z^m}z^m+Bu^m$ is bounded in $L^p(Q), p>1$. By parabolic regularity (see \cite{lady}), $S_2(A_0^{\tilde z^m}z^m+C_0^{\tilde z^m}z^m+Bu^m)$ is bounded in any $W^{2,1}_p,\forall p, 1<p<\infty$ (the space of anisotropic Sobolev functions). For $p$ big enough we have $W^{2,1}_p\subset C^{0,\alpha}(\overline Q)$ for some $0<\alpha<1$ (the space of H\"older continuous functions). $C^{0,\alpha}(\overline Q)$    is  compactly imbedded in $C(\overline Q)$. Consequently $(z^m)_m$ is a relatively compact sequence  in $L^\infty(Q)$.

	\item[iv)] $F_{z^0}$ is upper semi-continuous, \textit{i.e.} if $z^m\rightarrow z$, $\tilde z^m\rightarrow\tilde z$ in 
	$L^\infty$ and $z^m\in F_{z^0}(\tilde z^m)$ then $z\in F_{z^0}(\tilde z)$.
	
	Indeed  we have (see \eqref{linearizationcoef}) that $A_0^{\tilde z^m}\rightarrow A_0^{\tilde z}$, $C_0^{\tilde z^m}\rightarrow C_0^{\tilde z}$ in $L^\infty$ and as $ (z^m)_m$ is relatively compact in $C([0,T];[L^2(\Omega)]^{n+1})$ we may pas to the limit in \eqref{eqlim} and find that $z \in F_{z^0}(\tilde z ).$
\end{enumerate}

Now we conclude the proof by Kakutani fixed point theorem, which insures existence of $z\in \mathcal{ M}_{\beta_0}$     such that $z\in F_{z^0}(  z )$ \textit{i.e.} there exists $\overline u\in\mathcal{U^*}$ such that $z^{\overline u,z}=z$. In conclusion $y^{\overline u}:=\overline y+z$ is the solution to the controlled system \eqref{nonlinsystem} with control $\overline u$ satisfying $y^{\overline u}(T)=\overline y$.
\fin 

\section{Parabolic systems with tree-like couplings. Null controllability.}\label{sec-tree}

The case of tree-type couplings is more technical to describe in the context of the needed hypotheses on the supports of coupling functions or coupling coefficients in the linear models. These hypotheses are essential for the construction of appropriate auxiliary and weight functions in the corresponding Carleman estimates which are established  for each equation associated to a node in the graph, estimates which in the end should couple well into a global observability estimate. 

The hypotheses we impose to the supports of the coupling coefficients allow to treat each equation corresponding to a node of the tree as the center of a star-like system together with the directly actuated variables and corresponding equations. The star-like sub-graphs at the same level of the tree should be, in some sense, independently actuated. 
\medskip

We will say that a controlled linear parabolic  system  has a tree-type coupling in zero order terms if the system has the form:
\begin{equation}\label{linsystemtree}
\left\lbrace
\begin{array}{ll}
D_tz_{0}-\Delta z_0=c_0(t,x)z_0+\chi_{\omega_0}u, & \text{ in }(0,T)\times\Omega,\\
D_tz_{i}-\Delta z_i=a_{i\textbf{k}(i)}(t,x)z_{\textbf{k}(i)}+c_i(t,x)z_i,\, i\in\overline{1,n},\, &\text{ in }(0,T)\times\Omega,\\
z_0=...=z_{n}=0, &\text{ on }(0,T)\times\partial\Omega,\\
z(0,\cdot)=z^0,
\end{array}
\right.
\end{equation} 
with the following assumptions on the function $\textbf{k}:\{1,\ldots,n\}\rightarrow \{0,1,\ldots,n\}$:

\begin{equation}\label{k}
\forall i\in \{1,\ldots,n\},\exists m=m(i), 1\le m\le n-1, (\textbf{k}\circ)^m(i)=\textbf{k}\circ\ldots\circ\textbf{k}(i)=0.
\end{equation}

The linear   problem \eqref{linsystemtree} may be reformulated as: 
\begin{equation}\label{linsystree}
\left\lbrace
\begin{array}{ll}
D_tz=\textbf{A}z+\mathbf{A_0}(t)z+\mathbf{C}(t)z+\textbf{B}u, & t>0,  \\
z(0)=z^0,   &\\
\end{array}
\right.
\end{equation} 
where $\mathbf{C}(t)z=C_0(t,\cdot)z(\cdot)$ and $\mathbf{A_0}(t)z=A_0(t,\cdot)z(\cdot)$ with  $$C_0(t,x)= diag(c_i(t,x))_{i=\overline{0,n}},$$ and the coupling matrix  $$A_0(t,x)=(a_{il})_{i,l\in\overline{1,n}}=(a_{i\textbf{k}(i)}\delta_{l\textbf{k}(i)})_{i,l\in\overline{1,n}},$$
where we denoted by $\delta_{lj}$   the Kronecker symbol.
Denote by 
$$
\textbf{I}_j=\textbf{k}^{-1}(j)=\{i\in\overline{1,n}:\textbf{k}(i)=j\}.
$$
Fix now a family of open subsets  $\omega_i\subset\Omega,i\in\overline{1,n}$ such that
\begin{equation}
\label{controlierarhic1}
D_i:=\omega_i\cap\omega_{\textbf{k}(i)}\cap\cdots\cap\omega_{(\textbf{k}\circ)^{m(i)}}\ne\emptyset.
\end{equation}
\begin{equation}\label{controlierarhic2}
D_i\setminus\bigcup_{j\ne i,\textbf{k}(j)=\textbf{k(i)}}\omega_j\ne\emptyset.
\end{equation}
Choose further  a family of open subsets   $\{\underline\omega_j\}_{j\in\overline{0,n}}$ with the properties  
\begin{eqnarray}
\label{ci1}\underline\omega_0\subset\subset\omega_0,\quad \underline{\omega}_i\subset\subset D_i\setminus\bigcup_{l\ne i,\textbf{k}(l)=\textbf{k(i)}}\omega_l,\\
\label{ci2}\underline{\omega}_i\subset\subset\underline{\omega}_{\textbf{k}(i)}\subset\subset \underline\omega_0,\,i\in\overline{1,n}.
\end{eqnarray} 
For $M,\delta>0,$ and the family of open subsets described above $\{\underline\omega_i\}_i$, we introduce the following classes  of coefficients sets:
\begin{equation}\label{hyptree}
\begin{aligned}
&\mathcal{E}_{M,\delta,\{\underline\omega_i\}_i,\textbf{k}}=\biggl\{ E=\{a_{i\textbf{k}(i)},c_j\}_{i\in\overline{1,n},j\in\overline{0,n}}:a_{i\textbf{k}(i)},c_j\in L^\infty(Q),  \\
&\|a_{i\textbf{k}(i)}\|_{L^\infty}, \|c_j\|_{L^\infty}\leq M, a_{i\textbf{k}(i)}=0 \text{ in } Q\setminus Q_{\omega_i},\text{ and } |a_{i\textbf{k}(i)}|\ge\delta \text{ on }Q_{\underline\omega_i} \biggr\}.
\end{aligned}
\end{equation}

In order to study controllability  we consider the system adjoint to  system \eqref{linsystemtree}:
\begin{equation}\label{adjlinsystemtree}
\left\lbrace
\begin{aligned}
&-D_tp_{j}-\Delta p_j-c_j(t,x)p_j=\sum_{l,\textbf{k}(l)=j}a_{lj}(t,x)p_l=\mathcal{N}_{j}(t,x),\, j\in\overline{0,n},\, \text{ in }Q,\\
&p_0=...=p_{n}=0, \text{ on }(0,T)\times\partial\Omega,\\
\end{aligned}
\right.
\end{equation}
where for simplicity of further calculations we denoted by $$\mathcal{N}_{j}(t,x)=\sum_{l,\textbf{k}(l)=j}a_{lj}(t,x)p_l(t,x).$$
%In abstract form, the adjoint  to \eqref{linsystree} reads:
%\begin{equation}\label{adjlinsystree}
%-D_tp=\textbf{A}p+\mathbf{A_0^*}(t)p+\mathbf{C}(t)p,  t>0.
%\end{equation} 

As we have seen in the previous sections all controllability results have as essential ingredient an appropriate Carleman inequality for the adjoint system. For obtaing such estimates it is essential to have corresponding auxiliary functions which appear in the construction of the weights. We describe this in what follows

Consider again open subsets 
$$\tilde\omega_j\subset\subset\underline\omega_j, j\in\overline{0,n},$$ 
and auxiliary functions
$$\eta_j\in C^2(\overline{\Omega}),\, 0<\eta_j \text{ in }\Omega,\,\eta_j|_{\partial\Omega}=0,\{x\in\overline\Omega: |\nabla\eta_j(x)|=0\}\subset\subset\tilde\omega_j, j\in\overline{0,n}.$$
We construct now the weight functions entering the various Carleman es\-timates, with the following properties: 
\begin{enumerate}
	\item[i)]$\psi_{j,f}, j\in\overline{0,n},\textbf{I}_j\ne\emptyset$, $\psi_{i,s},i\in\overline{1,n}$ are defined by
	\begin{equation}\label{psi-star}
	\psi_{j,f}:=\eta_j+K_j,  \quad \psi_{i,s}:=\eta_i+\tilde K_i
	\end{equation}
	for some fixed positive constants $K_j,\tilde K_i>0$  and such that for a fixed $\epsilon>0$ we have
	\begin{equation}\label{psifpsis0}
	\psi_{i,s}>\sup_{\overline\Omega}\psi_{j,f}+2\epsilon, \forall i\in \textbf{I}_j,\,\textbf{I}_j\ne\emptyset;
	\end{equation}
	\begin{equation}\label{psifpsis1}
	\psi_{i,f}>\sup\{ \psi_{l,s}:\textbf{k}(l)=\textbf{k}(i)\}+2\epsilon, \forall i\in\overline{1,n},  \textbf{I}_i\ne\emptyset;
	\end{equation}
	
	\item[ii)]	\begin{equation}\label{psipsitree}
	\frac{\sup \psi_{j,f}}{\inf\psi_{j,f}}<\frac87, \frac{\sup \psi_{i,s}}{\inf\psi_{i,s}}<\frac87;
	\end{equation}
	\item[iii)] For $j\in\overline{0,n}$ such that $\textbf{I}_j\ne\emptyset$ we define
	\begin{eqnarray}
	\overline{\psi}_j=\sup\{\psi_{j,f}(x), \psi_{i,s}(x): i\in \textbf{I}_j, x\in\Omega\}+\epsilon,\\
		\underline{\psi}_j=\inf\{\psi_{j,f}(x), \psi_{i,s}(x): i\in \textbf{I}_j, x\in\Omega\}-\epsilon.
	\end{eqnarray}
	\item[iv)] Denote by $\overline\psi=\sup\{\overline\psi_j:\textbf{I}_j\ne\emptyset\}$ and $\underline\psi=\inf\{\underline\psi_j:\textbf{I}_j\ne\emptyset\}$ and 
	\begin{equation}\label{fialfabartree}
	\overline\varphi_j(t)=\overline\varphi_j^\lambda(t):=\frac{e^{\lambda\overline\psi_j}}{t(T-t)},\quad
	\overline\alpha_j(t)=\overline\alpha_j^\lambda(t):=\frac{e^{\lambda\overline\psi_j}-e^{1.5\lambda\overline\psi}}{t(T-t)},
	\end{equation}
	\begin{equation}\label{fialfabartree2}
	\underline\varphi_j(t)=\underline\varphi_j^\lambda(t):=\frac{e^{\lambda\underline\psi_j}}{t(T-t)},\quad
	\underline\alpha_j(t)=\underline\alpha_j^\lambda(t):=\frac{e^{\lambda\underline\psi_j}-e^{1.5\lambda\overline\psi}}{t(T-t)}.
	\end{equation}
	\begin{equation}\label{fialfabartree3}
	\underline\alpha(t)=\frac{e^{\lambda\underline\psi}-e^{1.5\lambda\overline\psi}}{t(T-t)},\, \overline\alpha(t)=\frac{e^{\lambda\overline\psi}-e^{1.5\lambda\overline\psi}}{t(T-t)}
	\end{equation}
\end{enumerate}

\begin{remark}
	\label{rem_weight_order}
	Observe that this construction of the weight functions allows saying that 
	$$
	\overline \psi_j<\underline\psi_i, i\in \textbf{I}_j, \textbf{I}_j\ne \emptyset,
	$$
	and thus, given  $\theta>0$  there exists $ s(\theta)$ such that for $s>s(\theta)$ we have
\begin{equation}
e^{s\overline\alpha_j(t)}\leq \theta  e^{s\underline\alpha_i(t)}, i\in \textbf{I}_j, \textbf{I}_j\ne \emptyset, t\in[0,T].
\end{equation}
\end{remark}

The Carleman estimates we establish now in the tree coupling case are given in the following theorem: 

\begin{theorem}\label{thCarltree} Suppose that the coupling coefficients in \eqref{adjlinsystemtree} satisfy $$\{a_{i\textbf{k}(i)},c_j\}_{i\in\overline{1,n},j\in\overline{0,n}}\in\mathcal{E}_{M,\delta,\{\underline\omega_i\}_i,\textbf{k}}.$$
	
	Then there exist constants $\lambda_0,s_{0}$ such that for $\lambda>\lambda_0$ there exists a constant $C>0$ depending on $(M,\delta,\{\underline\omega_i\}_i,\lambda)$
 such that, for any $ s\geq s_0$, the following inequality holds:
	\begin{equation}\label{estCarlemantree}\begin{aligned}
	&\int_Q(|D_tp|^2+|D^2p|^2+|D p|^2+|p|^2)e^{2s\underline\alpha}dxdt\\
	&\leq C\int_{Q_{\omega_0}}|p_0|^2e^{2s\overline\alpha}dxdt\\
	\end{aligned}
	\end{equation}
	for all $p\in H^1(0,T;L^2(\Omega))\cap L^2(0,T; H^2(\Omega))$ solution of \eqref{adjlinsystemtree}.
	
	Moreover, there exists $m_0\in\mathbb{N}$ and $\delta_1>0$  such that  we have the following $L^\infty-L^2$ Carleman estimate
	\begin{equation}\label{LinftyL2Carlemantree}
	\|p e^{(s+m_0\delta_1)\underline{\alpha}}\|_{L^{\infty}(Q)}\leq C \|p_0 e^{s\overline\alpha}\|_{L^2(Q_{\omega_0})}.
	\end{equation}
\end{theorem}
\proof 
For $j\in\overline{0,n}$ we write separately Carleman inequalities for the case $\textbf{I}_j\ne\emptyset$ and respectively for the case  $\textbf{I}_j=\emptyset$.
If $j\in\overline{0,n}$ is such that $\textbf{I}_j\ne\emptyset$ we treat the equations satisfied by  $p_j$ and $p_l,l\in \textbf{I}_j$ as a nonhomogeneous  adjoin system, as in the star-like couplings \eqref{withsource}, while in the case $\textbf{I}_j=\emptyset$ we have to deal with homogeneous parabolic equations:
\begin{equation}\label{adjstarnonhom}
\left\lbrace
\begin{array}{ll}
-D_tp_{j}-\Delta p_j-c_j(t,x)p_j=\sum_{l,\textbf{k}(l)=j}a_{lj}(t,x)p_l,\, &\text{ in }(0,T)\times\Omega,\\
-D_tp_{l}-\Delta p_l-c_l(t,x)p_l=\mathcal{N}_{l}(t,x),&l\in \textbf{I}_j.
\end{array}
\right.
\end{equation}
For the case $\textbf{I}_j\ne\emptyset$ a Carleman estimate, which is an immediate consequence to intermediate estimate \eqref{estCarleman1},  states that there exists $\overline s_j$ and $C>0$ not depending on $s$ such that for $s>\overline s_j$ we have

	\begin{equation}\label{estCarlemantree1}\begin{aligned}
&\int_Q(|D_tp_j|^2+|D^2p_j|^2+|D p_j|^2+|p_j|^2)e^{2s\underline\alpha_j}dxdt\\
&+\int_Q\left[\sum_{i\in\textbf{I}_j}(|D_tp_i|^2+|D^2p_i|^2+|D p_i|^2+|p_i|^2)\right]e^{2s\underline\alpha_j}dxdt\\
&\leq C\left[\int_{Q_{\underline\omega_j}}|p_j|^2e^{2s\overline\alpha_j}dxdt+\sum_{i\in\textbf{I}j}\int_{Q_{\underline\omega_i}}|p_i|^2e^{2s\overline\alpha_j}\right]\\
&+C\sum_{i\in\textbf{I}_j}\int_{Q}|\mathcal{ N}_i(t,x)|^2e^{2s\overline\alpha_j}dxdt\\
	&  \leq C\left[\int_{Q_{\underline\omega_j}}|p_j|^2e^{2s\overline\alpha_j}dxdt+\sum_{i\in\textbf{I}_j}\int_{Q_{\underline\omega_i}}|p_i|^2e^{2s\overline\alpha_i}\right]dxdt\\
	&+C\sum_{i\in\textbf{I}_j,l\in \textbf{I}_i}\int_{Q}\theta|p_l(t,x)|^2e^{2s\underline\alpha_l}dxdt,
\end{aligned}
\end{equation}
where we have used Remark \ref{rem_weight_order} in order to say that $e^{2s\overline\alpha_j}\le\theta e^{2s\underline\alpha_i}\le \theta e^{2s\underline\alpha_l} $ for $\theta>0$  to be fixed later and  $s>s(\theta)$ big enough.
\medskip

In the case $\textbf{I}_j=\emptyset$, we write the Carleman estimate for the homogeneous equation$$-D_tp_{j}-\Delta p_j-c_j(t,x)p_j=0.$$
So, there exist constants $\overline s_j>0$ and $C>0$ such that for $s>\overline s_j$ 

\begin{equation}\label{estCarlemantree2}\begin{aligned}
&\int_Q(|D_tp_j|^2+|D^2p_j|^2+|D p_j|^2+|p_j|^2)e^{2s\underline\alpha_j}dxdt\\
&  \leq C\int_{Q_{\underline\omega_j}}|p_j|^2e^{2s\overline\alpha_j}dxdt.
\end{aligned}
\end{equation}
We add now estimates \eqref{estCarlemantree1} and \eqref{estCarlemantree2} and we obtain for some constant $C>0$ and $s>\max_j\overline s_j$:
\begin{equation}\label{estCarlemantree3}
\begin{aligned}
&\sum_{j\in\overline{0,n}}\int_Q(|D_tp_j|^2+|D^2p_j|^2+|D p_j|^2+|p_j|^2)e^{2s\underline\alpha_j}dxdt\leq\\
&C\left[\sum_{j\in\overline{0,n}}\int_{Q_{\underline\omega_j}}|p_j|^2e^{2s\overline\alpha_j}dxdt+
\sum_{j\in\overline{1,n}}\int_{Q}\theta|p_j(t,x)|^2e^{2s\underline\alpha_j}\right]dxdt.
\end{aligned}
\end{equation}
Choosing $\theta$  small enough we see that the integrals on $Q$ in the right side may be absorbed in the left side of the inequality and obtain 
\begin{equation}\label{estCarlemantree4}
\begin{aligned}
&\sum_{j\in\overline{0,n}}\int_Q(|D_tp_j|^2+|D^2p_j|^2+|D p_j|^2+|p_j|^2)e^{2s\underline\alpha_j}dxdt\\
&\leq C\sum_{j\in\overline{0,n}}\int_{Q_{\underline\omega_j}}|p_j|^2e^{2s\overline\alpha_j}dxdt.
\end{aligned}
\end{equation}
Observe now that for $j\ge 0$, by \eqref{k} there exists $m=m(j)$ and the sequence $j_0=j,j_1=\textbf{k}(j_0),\ldots,j_m=(\textbf{k}\circ)^m(j)=0$. Now, by \eqref{controlierarhic1}, \eqref{controlierarhic2}, \eqref{ci1}, \eqref{ci2}, and looking only to the subdomains $\underline \omega_{j_l}, l\in\overline{0,m}$ we find a sequence of equations for $l\in\overline{0,m-1}$, forming cascade like system:
\begin{equation}
-D_tp_{j_{l+1}}-\Delta p_{j_{l+1}}-c_{j_{l+1}}(t,x)p_{j_{l+1}}=a_{j_l,j_{l+1}}(t,x)p_{j_l},\, \text{ in }(0,T)\times\underline\omega_{j_{l+1}}.
\end{equation}
Now, as $\underline\omega_{j_l}\subset\subset\underline\omega_{j_{l+1}}$ we find, as in the \S\ref{secobservcarl}
\begin{equation}
\int_{Q_{\underline\omega_{j_{l}}}} |p_{j_{l}}|^2e^{2s\overline\alpha_{j_l}}dxdt\le C     \int_{Q_{\underline\omega_{j_{l+1}}}}|p_{j_{l+1}}|^2e^{2s\overline\alpha_{j_{l+1}}}dxdt.
\end{equation}
Consequently, for all $j\in\overline{1,n}$ we find, by coupling the chain estimates above, that
\begin{equation}
\int_{Q_{\underline\omega_{j}}} |p_{j }|^2e^{2s\overline\alpha_{j }}dxdt\le C \int_{Q_{\underline\omega_{0}}} |p_{0 }|^2e^{2s\overline\alpha_{0 }}dxdt,
\end{equation}
which plugged into \eqref{estCarlemantree4} gives a final Carleman estimate 
\begin{equation}\label{estCarlemantree5}
\begin{aligned}
&\sum_{j\in\overline{0,n}}\int_Q(|D_tp_j|^2+|D^2p_j|^2+|D p_j|^2+|p_j|^2)e^{2s\underline\alpha_j}dxdt\\
&\leq C\int_{Q_{\underline\omega_0}}|p_0|^2e^{2s\overline\alpha_0}dxdt.
\end{aligned}
\end{equation}
which gives the final conclusion in the $L^2-L^2$ framework, \eqref{estCarlemantree}.
 
 The $L^\infty-L^2$ estimate \eqref{LinftyL2Carlemantree} follows by the same lines in the corresponding Theorem \ref{th1obs}, using the bootstrap argument in connection to the regularity properties of the parabolic flow.\fin

The main result concerning controllability with one control for linear parabolic systems with tree-like couplings is the following:
\begin{theorem}\label{thcontrol_lin_tree}
	Consider system \eqref{linsystemtree} with coefficients in $\mathcal{\tilde E}_{M,\delta,\{\underline\omega_i\}_i}$. 
	Then there exists a constant $C=C({M,\delta,\{\underline\omega_i\}_i})$ such that for all  $z^0\in  H$ there exists $u^*\in L^2(0,T;L^2(\omega_0))\cap L^\infty(Q_{\omega_0})$ which drives the corresponding solution to   \eqref{linsystemtree} in $0$, \textit{i.e.} $z=z^{u^*}$  satisfies $z(T)=0$ and the control satisfies the norm estimate 
	\begin{equation}\label{est_contr_l2_tree}
	\|u^*e^{-s\overline\alpha}\|_{L^2(0,T;L^2(\omega_0))}+	\|u^* \|_{L^\infty(Q_{\omega_0})}\leq C\|z^0\|_{L^2(\Omega)}.
	\end{equation}
\end{theorem}

\proof
The proof is identical to the proof of Theorem \ref{thcontrol} by using the Carleman estimates for the linear adjoint system \eqref{adjlinsystemtree} given by  Theorem \ref{thCarltree} and a corresponding observability estimate as the one given by Remark \ref{rem_obs}.

Note here that for the $L^\infty$ estimate on the control, one needs to use in Carleman estimate a parameter $\lambda$ such that 
\eqref{ordineaponderilor2} holds.
\fin

\medskip

Controllability of nonlinear semilinear parabolic systems with tree-like couplings may be studied in analogy to the star-like case. We consider semilinear systems of parabolic equations, with tree type couplings in zero order terms, of the form
\begin{equation}\label{nonlinsystemtree}
\left\lbrace
\begin{array}{ll}
D_ty_{0}-\Delta y_0=\overline g_0(x)+f_0(x,y_0)+\chi_{\omega_0}u, & \text{ in }(0,T)\times\Omega,\\
D_ty_{i}-\Delta y_i=\overline g_i(x)+f_i(x,y_{\textbf{k}(i)},y_i),\, i\in\overline{1,n},\, &\text{ in }(0,T)\times\Omega,\\
y_0=...=y_{n}=0, &\text{ on }(0,T)\times\partial\Omega,\\
y(0,\cdot)=y^0,&
\end{array}
\right.
\end{equation} 
where $\overline g_j\in L^\infty(\Omega),\,  j\in\overline{0,n}$ and  $\overline y=(\overline{y}_0,...,\overline{y}_n)\in [L^\infty(\Omega)]^{n+1}$ is a corresponding stationary solution.

We assume the following hypotheses on the nonlinearities:
\begin{enumerate}
	\item[\textit{(H1')}] $f_0\in C^1(\Omega\times\R), f_i\in C^1(\Omega\times\R\times\R),i\in\overline{1,n} $  there exist $\omega_1,...\omega_{n}\subset \Omega$  open nonempty subsets of
	$\Omega$ satisfying \eqref{controlierarhic1},\eqref{controlierarhic2} and
	\begin{equation}
	(\omega_i\cap\omega_{\textbf{k}(i)})\setminus\bigcup_{j\ne i, \textbf{k}(j)=\textbf{k}(i)}\omega_j\neq\emptyset,\,\forall i\in\overline{1,n},
	\end{equation}
	and for all $i\in\overline{1,n}$ we have
	\begin{equation}\label{fsuport1}
	f_i(x,\tau,\xi)=0 \,\forall x\in\Omega\setminus\omega_i,\, \tau,\xi\in\R;
	\end{equation}
	\item[\textit{(H2')}]For a family of subdomains $\{\underline\omega_i\}_i$ satisfying \eqref{ci1},\eqref{ci2}, by defining for $i\in\overline{1,n}$ the coefficients
	 $$a^0_{i\textbf{k}(i)}(x):=\frac{\partial f_i}{\partial y_{\textbf{k}(i)}}(x,\overline{y}_{\textbf{k}(i)}(x),\overline{y}_i(x))$$
	$$
c^0_{0}(x):=\frac{\partial f_0}{\partial y_0}(x,\overline{y}_0(x)),\,	c^0_{i}(x):=\frac{\partial f_i}{\partial y_i}(x,\overline{y}_{\textbf{k}(i)}(x),\overline{y}_i(x)),
	$$
	
	we assume that for some $M_0,\delta_0>0$ we have
	\begin{equation}
	\{a^0_{i\textbf{k}(i)},c^0_j\}_{i\in\overline{1,n},j\in\overline{0,n}}\in\mathcal{E}_{M_0,\delta_0,\{\underline\omega_i\}_i,\textbf{k}}.
	\end{equation}

\end{enumerate}

\begin{theorem}\label{th_local_contr_tree}
	Suppose $\overline y$ is a stationary state to uncontrolled ($u=0$) \eqref{nonlinsystemtree} and that  functions $f_j,j\in\overline{0,n}$ satisfy hypotheses \textit{(H1')}, \textit{(H2')}. Then, for all $\beta_0>0$ there exist $\zeta_0=\zeta_0(\beta_0)>0$ and $C=C(\beta_0,\{\underline \omega_i\}_i,\overline{y})$ such that if $\|y^u(0)-\overline y\|<\zeta_0$ there exists a control $u\in L^\infty(Q)$ satisfying
	$$
	\|u\|_{L^\infty(Q)}\le C\|y^u(0)-\overline y\|_{L^\infty(\Omega)}
	$$
	and 
	$$y^u(T,\cdot)=\overline y,$$
	with
	$$ \|y(t,\cdot)-\overline y\|_{L^\infty}\le \beta_0,\,t\in[0,T]. $$
\end{theorem}

\begin{remark}
\begin{enumerate}
\item Our results remain valid if instead of the  operator $\Delta$ we use  general elliptic operators which may be differently chosen in each of the equation of the system: 
\begin{equation}
L_iy_i:=-\sum_{j,k=1}^N D_j(\alpha^{jk}_i D_ky_i)+\sum_{k=1}^N\beta^{k}_i D_ky_i+\gamma_iy_i\quad i=\overline{1,n},
\end{equation}
with general boundary conditions which may be also of  Neumann or Robin type. Here $ (\alpha^{jk}_i)_{j,k}$ satisfy uniform ellipticity conditions in $\Omega$. In our study we  need also to impose  regularity assumptions on the coefficients ( $\alpha^{jk}_i\in W^{1,\infty}(\Omega),\beta^{k}_i, \gamma_i\in L^\infty(\Omega)$); these regularity assumptions allow the development of the bootstrap argument based on the regularizing properties of the parabolic flow when establishing an $L^\infty$ framework for the controllability problem.

	%\item  De pus primul sistem si de spus ca nu e controlabil (ala e star like) inclusiv scri care e matricea Kalman.
\item The hypotheses on the support of the coupling coefficients is essential for our approach to the controllability problem.  In fact, for the systems we consider with the same type of couplings but with constant coupling coefficients controllability no longer occurs. 
take for example the following system with a star-type coupling ($\alpha$ and $\beta$ are fixed real constants):

\begin{equation}\label{rem_ex1}
\left\lbrace
\begin{array}{ll}
D_tz_{0}-\Delta z_0=\chi_{\omega_0}u, & \text{in }(0,T)\times\Omega,\\
D_tz_{1}-\Delta z_1=\alpha z_0,\,  &\text{in }(0,T)\times\Omega,\\
D_tz_{2}-\Delta z_2=\beta z_0,\,  &\text{in }(0,T)\times\Omega,\\
z_0=z_1=z_{2}=0, &\text{on }(0,T)\times\partial\Omega.\\
\end{array}
\right.
\end{equation}
Considering the results in \cite{khodja_burgos2},\cite{khodja_burgos1}, null controllability occurs if and only if the Kalman rank condition $\text{rank} [A_0|B]=3.$ However, in this situation the Kalman matrix is
$[A_0|B]=
\begin{pmatrix}
1&0&0\\
0&\alpha&0\\
0&\beta&0
\end{pmatrix}$
and  its rank is  $2$.

Also, if we consider  the parabolic system with tree-like couplings \eqref{linsystem_tree} in  \S2 Preliminaries, with constant coefficients $c_j=0$, $a_{10}=a_{20}=a_{31}=a_{41}=1$, the  Kalman matrix   

 $[A_0|B]=
\begin{pmatrix}
1&0&0&0&0\\
0&1&0&0&0\\
0&1&0&0&0\\
0&0&1&0&0\\
0&0&1&0&0\\
\end{pmatrix}$
and this has rank $3$; thus the system is not null controllable.

In fact one may see the results in this paper more as an extension of the results concerning cascade-like parabolic systems with nonconstant coefficients (see \cite{teresa_burgos}).  
\end{enumerate}
\end{remark}

\end{document}